\documentclass[a4paper]{amsart}
\usepackage[
left=28truemm,
right=28truemm
]{geometry}
\usepackage{amsfonts, amssymb, amsmath, verbatim, amsthm, mathrsfs, amscd, enumerate, ascmac, fancyhdr, caption, hyperref, framed, subfigure}
\usepackage{bbm}
\numberwithin{equation}{section}
\usepackage{tikz}
\usetikzlibrary{positioning,intersections, calc, arrows,decorations.markings, angles}
\theoremstyle{plain}
\newtheorem{theorem}{Theorem}[section]
\newtheorem{lemma}[theorem]{Lemma}
\newtheorem*{lemma*}{Lemma}
\newtheorem{proposition}[theorem]{Proposition}
\newtheorem*{proposition*}{Proposition}
\newtheorem{corollary}[theorem]{Corollary}
\theoremstyle{definition}

\newtheorem*{claim*}{Claim}
\newtheorem*{remarks}{Remarks}

\theoremstyle{remark}
\newtheorem*{remark}{Remark}
\newtheorem*{Acknowledgements}{Acknowledgements}

\def\supp{\mathrm{supp\,}}
\def\d{\mathrm{d}}
\def\ae{\mathrm{a.e.}}
\def\dim{\mathrm{dim}}
\def\Re{\mathfrak{R}}

\def\C{\mathcal{C}}

\def\1{\mathbbm 1}

\def\C{\mathcal{C}}

\def\1{\mathbbm 1}

\newcommand{\R}{{\mathbb R}}
\newcommand{\Z}{{\mathbb Z}}

\newcommand{\N}{{\mathbb N}}
\newcommand{\Co}{{\mathbb C}}

\newcommand{\F}{\mathcal{F}}
\newcommand{\J}{\mathcal{J}}
\newcommand{\I}{\mathcal{I}}

\let\Re\relax
\DeclareMathOperator{\Re}{Re}

\author[Bez]{Neal Bez}
\address[Neal Bez]{Department of Mathematics, Graduate School of Science and Engineering,
Saitama University, Saitama 338-8570, Japan}
\email{nealbez@mail.saitama-u.ac.jp}

\author[Kinoshita]{Shinya Kinoshita}
\address[Shinya Kinoshita]{Department of Mathematics, Tokyo Institute of Technology, Meguro-ku, Tokyo, 152-8551, Japan}
\email{kinoshita@math.titech.ac.jp}

\author[Shiraki]{Shobu Shiraki}
\address[Shobu Shiraki]{Departamento de Matem\'atica, Instituto Superior T\'ecnico, Av. Rovisco Pais, Lisboa, 1049-001, Portugal}
\email{shobushiraki@tecnico.ulisboa.pt}

\thanks{
This work was supported by JSPS Kakenhi grant numbers 19H00644, 19H01796, 22H00098 and 23H01080 (Bez), 
21J00514, 22KJ0446 (Kinoshita), and 19H01796, 22H00098 (Shiraki). The third author is also supported by Centro de Análise Matemática, Geometria e Sistemas Dinâmicos (CAMGSD)
}

\begin{document}
\date{\today}
\title[
]
{
Boundary Strichartz estimates and pointwise convergence for orthonormal systems
}

\subjclass[2010]{}
\begin{abstract}
We consider maximal estimates associated with fermionic systems. First we establish maximal estimates with respect to the spatial variable. These estimates are certain boundary cases of the many-body Strichartz estimates pioneered by Frank, Lewin, Lieb and Seiringer. We also prove new maximal-in-time estimates, thereby significantly extending work of Lee, Nakamura and the first author on Carleson’s pointwise convergence problem for fermionic systems. 
\end{abstract}

\maketitle

\section{Introduction}

\subsection{Background and motivation}

Strichartz estimates for orthonormal systems of initial data take the form
\begin{equation} \label{e:ONSqr}
\bigg\| \sum_j \lambda_j |Uf_j|^2 \bigg\|_{L^{\frac{q}{2}}_tL^\frac{r}{2}_x(\mathbb{R} \times \mathbb{R}^d)} \lesssim \| \lambda \|_{\ell^\beta}.
\end{equation}
As usual, the notation $\lesssim$ will be used to absorb constants which may depend on the dimension $d$ and exponents such as $q,r$ and $\beta$. Here we are primarily interested in the Schr\"odinger equation and $Uf(t,x) = e^{it\Delta}f(x)$. The initial data $(f_j)_j$ is a family of orthonormal functions in the homogeneous Sobolev space $\dot{H}^{s}(\mathbb{R}^d)$ with (necessarily) $s = \frac{d}{2} - \frac{d}{r} - \frac{2}{q}$, and $\lambda = (\lambda_j)_j$ is a sequence of scalars. The study of \eqref{e:ONSqr} was pioneered by Frank--Lewin--Lieb--Seiringer \cite{FLLS} and the estimates have been used to develop a rigorous understanding of the dynamics of a system of infinitely many fermions. For such applications, we refer the reader to, for example, papers by Lewin--Sabin \cite{LewinSabin1, LewinSabin2} and Frank--Sabin \cite{FS_AJM}, as well as very closely related work of Chen--Hong--Pavlovi\'c \cite{CHP1, CHP2}.

For fermionic particles, the orthonormality assumption is natural. This is tied together with the Pauli exclusion principle which forbids two fermions occupying the same quantum state. The dynamics of a system of $N$ fermions can be modelled by a system of coupled equations of Hartree type
\[
i\partial_t u_k + \Delta u_k = (w * \rho)u_k, \quad u_k(0) = f_k
\]
where $k = 1,\ldots,N$, $\rho(t,x) = \sum_{j=1}^N |u_j(t,x)|^2$ is the total density of particles, and $w$ is the interaction potential. To consider the case of infinitely many particles, it is convenient to work with density matrices and one is led to the equation (often called the reduced Hartree--Fock equation)
\begin{equation} \label{e:vonNeumann}
i\partial_t \gamma + [\Delta,\gamma] = [w * \rho_\gamma,\gamma], \quad \gamma(0) = \gamma_0,
\end{equation}
where $\rho_\gamma$ is the density function of the operator $\gamma$. Roughly speaking, Strichartz estimates of the form \eqref{e:ONSqr} (and their close cousins) have played a key role in the recent breakthroughs in giving a rigorous meaning to solutions of \eqref{e:vonNeumann} when the initial data $\gamma_0$ is not of trace class (\cite{CHP1, CHP2, FS_AJM, LewinSabin1, LewinSabin2}). Further recent developments in this direction include the study of ground states \cite{FGL, GLN} and a version of Carleson's pointwise convergence problem \cite{BLN_Selecta}. 

The pointwise convergence problem in particular provided an especially important source of motivation for the present paper and our focus here is on \emph{maximal estimates}, either with respect to the spatial variable or the temporal variable. Maximal-in-space estimates correspond to Strichartz estimates \eqref{e:ONSqr} with $r = \infty$ and this delicate case has, to a large extent, been left open. Maximal-in-time estimates correspond to variants of \eqref{e:ONSqr} with a $L^{\frac{q}{2}}_x L^\infty_t$ mixed-norm on the left-hand side (possibly in localised form) and yield pointwise convergence of $\gamma$ to $\gamma_0$ at the level of the density functions. As mentioned above, such estimates were first considered in \cite{BLN_Selecta} and here we significantly develop this line of investigation by understanding the effect of adding smoothness to the initial data $\gamma_0$ and also, in the spirit of work of Sj\"ogren--Sj\"olin \cite{SS89} and Barcel\'o--Bennett--Carbery--Rogers \cite{BBCR} (in the classical single-particle case), to provide information on the size of the so-called divergence sets where pointwise convergence to the initial data fails to hold. Our new results will appear in Section \ref{section:mainresults}; in advance of that, it will be helpful to include a more detailed discussion of the known results regarding the Strichartz estimates \eqref{e:ONSqr} and the maximal-in-time estimates.

\subsection{Prior work on \eqref{e:ONSqr}} 
Firstly, we observe that the case $\beta = 1$ is equivalent to the classical Strichartz estimate
\begin{equation} \label{e:classicalS}
\| Uf \|_{L^q_tL^r_x} \lesssim \|f\|_{\dot{H}^s}.
\end{equation}
Indeed, given \eqref{e:ONSqr} and taking all but one of the sequence $(\lambda_j)_j$ to be zero (and rescaling), we obtain \eqref{e:classicalS}. Conversely, from the triangle inequality and an application of \eqref{e:classicalS} for each $f_j$, one quickly obtains \eqref{e:ONSqr} with $\beta = 1$.  Notice that the latter argument makes no use of the orthogonality of the $f_j$ and thus one would like to understand how much gain (if any) can be sought from the orthogonality by raising $\beta \geq 1$ as far as possible. It turns out that the optimal value of $\beta$ depends on $q$ and $r$ in an interesting way and we now describe the known results in this direction. In order to do so, we first consider the case $r < \infty$ and divide the discussion into the so-called \emph{sharp admissible} and \emph{non-sharp admissible} cases, and following that focus on the case $r = \infty$ (with a brief discussion of the other boundary cases $q = 2,\infty$).

\subsubsection{The sharp admissible case with $r < \infty$} 

Following standard terminology, when $\frac{1}{q} = \frac{d}{2}(\frac{1}{2} - \frac{1}{r})$ holds, we shall refer to this case as the \emph{sharp admissible} case; here the Sobolev exponent $s$ coincides with zero and the initial data belong to $L^2(\mathbb{R}^d)$.

First, let us consider the case $d \geq 3$. The Keel--Tao endpoint corresponds to $(q,r) = (2,\frac{2d}{d-2})$ and it was shown in \cite{KeelTao} that \eqref{e:classicalS} holds. This answered a long-standing question about the validity of the endpoint case and interpolation with the easy estimate at $(q,r) = (\infty,2)$ (i.e. conservation of energy) yields all possible Strichartz estimates \eqref{e:classicalS} in the sharp admissible case. We refer the reader to  \cite{KeelTao} for discussion and references for earlier work on the non-endpoint cases. 

Somewhat curiously, at the Keel--Tao endpoint $(q,r) = (2,\frac{2d}{d-2})$, the optimal value of $\beta$ for \eqref{e:ONSqr} is 1 and there is in fact no room to extract any gain from the orthogonality of the $f_j$. This phenomenon was observed by Frank and Sabin \cite{FS_Survey} in which, more generally, they established  that \eqref{e:ONSqr} fails if $\beta > \frac{q}{2}$. Earlier, Frank \emph{et al.} \cite{FLLS} showed that \eqref{e:ONSqr} also fails if $\beta > \frac{2r}{r + 2}$. These thresholds coincide when $r = \frac{2(d + 1)}{d-1}$ and this turns out to be the endpoint case in the sense that the estimate \eqref{e:ONSqr} is known to be true with
\begin{equation} \label{e:FSlower}
\text{$\beta \leq \frac{2r}{r + 2}$ when $r \in \bigg[2,\frac{2(d + 1)}{d-1}\bigg)$}
\end{equation}
and
\begin{equation} \label{e:FSupper}
\text{$\beta < \frac{q}{2}$ when $r \in \bigg[\frac{2(d + 1)}{d-1},\frac{2d}{d-2}\bigg)$.}
\end{equation}
The estimates in \eqref{e:FSlower} for $r \in [2,\frac{2(d + 2)}{d}]$ were first established in \cite{FLLS} and later extended to $r \in [2,\frac{2(d + 1)}{d-1})$ in \cite{FS_AJM}. As observed in \cite{FS_Survey}, the estimates in \eqref{e:FSupper} follow from those in \eqref{e:FSlower} by an interpolation argument between \eqref{e:ONSqr} with $\beta = 1$ at the Keel--Tao endpoint $r = \frac{2d}{d-2}$ and estimates from \eqref{e:FSlower} with $r$ less than, but arbitrarily close to\footnote{This argument leaves open the case $\beta = \frac{q}{2}$ in \eqref{e:FSupper} and, as far as we are aware, this remains a very interesting open problem. This would follow if we could extend \eqref{e:FSlower} to $r = \frac{2(d+1)}{d-1}$ but unfortunately, as shown in \cite{FLLS}, such an estimate is false.}, the exponent $\frac{2(d+1)}{d-1}$.

When $d = 2$, the known results are of a somewhat similar nature; for instance, \eqref{e:FSlower} holds without modification (due to \cite{FLLS, FS_AJM}). However, at the Keel--Tao exponent we have $r = \infty$ and it is known that the classical Strichartz estimate
\begin{equation} \label{e:classicalSfalsed=2}
\| Uf \|_{L^2_tL^\infty_x(\mathbb{R} \times \mathbb{R}^2)} \lesssim \|f\|_{L^2(\mathbb{R}^2)}
\end{equation}
is false; see, for example, Montgomery-Smith \cite{MontSmith}. However, one may still obtain \eqref{e:FSupper} as it stands for $d = 2$ (in particular, not including $r = \infty$) by using \eqref{e:classicalS} with finite and sufficiently large values of $r$.

Finally, when $d = 1$, the estimates in \eqref{e:FSlower} again hold as stated (thanks to \cite{FLLS, FS_AJM}) and since the situation in \eqref{e:FSupper} does not arise, this completes the picture for the sharp admissible cases where $r < \infty$. 

\subsubsection{The non-sharp admissible case with $r < \infty$}
We refer to the case $\frac{1}{q} < \frac{d}{2}(\frac{1}{2} - \frac{1}{r})$ as the \emph{non-sharp admissible} case. As long as $q, r < \infty$, the results in \cite{FLLS, FS_AJM, FS_Survey} were extended to the non-sharp admissible case in \cite{BHLNS} (see also \cite{BLN_Forum}). To state the result, we introduce the notation $\beta(q,r)$ for the exponent satisfying the relation
\[
\frac{d}{2\beta(q,r)} = \frac{1}{q} + \frac{d}{r},
\]
and observe that $\beta(q,r) = \frac{2r}{r + 2}$ in the sharp admissible case.
\begin{theorem}[\cite{BHLNS}] \label{t:BHLNS}
Let $d \geq 1$, $q,r \in [2,\infty)$, and $\frac{1}{q} < \frac{d}{2}(\frac{1}{2} - \frac{1}{r})$. Then the estimate \eqref{e:ONSqr} holds in each of the following cases.

\emph{(i)} $\beta \leq \beta(q,r)$ when $\frac{d}{r} > \frac{d-1}{q}$.

\emph{(ii)} $\beta < \frac{q}{2}$ when $\frac{d}{r} \leq \frac{d-1}{q}$.
\end{theorem}
It was also shown in \cite{BHLNS} that 
\[
\beta \leq \min  \bigg\{ \beta(q,r), \frac{q}{2} \bigg\}
\]
is a necessary condition for \eqref{e:ONSqr} to hold, and thus Theorem \ref{t:BHLNS} is close to optimal.

\subsection{The boundary cases}
We shall refer the cases in the admissible region $\frac{1}{q} \leq \frac{d}{2}(\frac{1}{2} - \frac{1}{r})$, $q,r \in [2,\infty]$, when $q = 2$, $q = \infty$ and $r = \infty$ as boundary cases. As we have already noted, when $q = 2$ the only available estimates of the form \eqref{e:ONSqr} are when $\beta = 1$. 

The case $q = \infty$ is interesting and discussing this case will also naturally lead us to the wider context of extending classical estimates to the setting of orthonormal systems. To obtain \eqref{e:ONSqr} with $q = \infty$, one may invoke Lieb's version of the Sobolev inequality\footnote{In fact, Lieb \cite{Lieb_Sobolev} proved a somewhat stronger estimate with $\|\lambda\|_\infty^{1 - \frac{2}{r}} \|\lambda \|_1^{\frac{2}{r}}$ on the right-hand side of \eqref{e:ONSobolev}.}
\begin{equation} \label{e:ONSobolev}
\bigg\| \sum_j \lambda_j |f_j|^2 \bigg\|_{L^{\frac{r}{2}}(\mathbb{R}^d)} \lesssim \| \lambda\|_{\beta}
\end{equation}
for orthonormal systems $(f_j)_j$ in $\dot{H}^s(\mathbb{R}^d)$, where $r \in (2,\infty)$, $s = \frac{d}{2} - \frac{d}{r}$ and $\beta < \frac{r}{2}$. Using \eqref{e:ONSobolev}, and the obvious fact that orthonormality in $\dot{H}^s(\mathbb{R}^d)$ is preserved under the Schr\"odinger flow $e^{it\Delta}$ for each fixed $t \in \mathbb{R}$, we obtain \eqref{e:ONSqr} when $q = \infty$, $r \in (2,\infty)$ and $\beta < \frac{r}{2}$. Moreover, as shown in \cite[Proposition 7.1]{BHLNS},  this result cannot be extended to $\beta = \frac{r}{2}$ (even with a weak-type norm $L^{\frac{r}{2},\infty}_x$ on the left-hand side). 
Since $\beta(\infty,r) = \frac{r}{2}$, this gives a complete picture for the case $q = \infty$.

At this point, we digress very slightly, and note that prior to the appearance of \eqref{e:ONSobolev}, a few years earlier Lieb and Thirring \cite{LiebThirring} established an inequality of the same spirit associated with the Gagliardo--Nirenberg--Sobolev inequality. This inequality, referred to as the \emph{Lieb--Thirring inequality}, was key to their proof of stability of matter; in addition to \cite{LiebThirring}, we also refer the reader to \cite{Lieb_BAMS} for further details. For wider discussion on the pursuit of obtaining versions of classical inequalities for orthonormal systems, and further examples, we direct the interested reader to \cite{Frank_ICM, Frank_9, FrankSabin_Adv, Nakamura, Nguyen}.

Returning to \eqref{e:ONSqr} in the boundary case $r = \infty$, let us first point out that a  resolution of the boundary cases of the classical Strichartz estimates \eqref{e:classicalS} has only very recently been completed thanks to the work of Guo--Li--Nakanishi--Yan \cite{GLNY}. In particular, it is shown in \cite{GLNY} that the estimate \eqref{e:classicalSfalsed=2} is also false in higher dimensions; that is
\begin{equation} \label{e:classicalSfalse}
\| Uf \|_{L^2_tL^\infty_x(\mathbb{R} \times \mathbb{R}^d)} \lesssim \|f\|_{\dot{H}^{\frac{d-2}{2}}(\mathbb{R}^d)}
\end{equation}
fails for all $d \geq 3$. We also remark that, for any $d \geq 1$,
\begin{equation} \label{e:classicalSfalse_allq}
\| Uf \|_{L^q_tL^\infty_x(\mathbb{R} \times \mathbb{R}^d)} \lesssim \|f\|_{\dot{H}^{\frac{d}{2}-\frac{2}{q}}(\mathbb{R}^d)}
\end{equation}
is also false when $q = \infty$ (this follows from the well-known failure of the corresponding Sobolev embedding estimate), but holds for $q \in (2,\infty)$ when $d \geq 2$, and holds for $q \in [4,\infty)$ when $d = 1$ (see, for example, \cite[Section 2]{GLNY}).

For orthonormal systems, very few results regarding \eqref{e:ONSqr} are currently available in the case $r = \infty$. Here $d = 1$ is somewhat special in the sense that there is a sharp-admissible case (i.e. $(q,r) = (4,\infty)$) and whether one can go beyond $\beta = 1$ was raised by Frank--Sabin \cite{FS_Survey}. In this specific case, it seems natural to conjecture that
\begin{equation} \label{e:d=1boundaryend}
\bigg\| \sum_j \lambda_j |Uf_j|^2 \bigg\|_{L^{2}_tL^\infty_x(\mathbb{R} \times \mathbb{R})} \lesssim \| \lambda \|_{\ell^\beta}
\end{equation}
holds for $\beta < 2$. However, as far as we are aware, whether this holds or not is a challenging open problem. The weak-type version
\begin{equation}\label{e:d=1boundaryend_weak}
\bigg\| \sum_j \lambda_j |Uf_j|^2 \bigg\|_{L^{2,\infty}_tL^\infty_x(\mathbb{R} \times \mathbb{R})} \lesssim \| \lambda \|_{\ell^\beta}
\end{equation}
for $\beta < 2$ was established very recently in \cite{BLN_Selecta}, but it appears to be non-trivial to upgrade this to a strong-type estimate\footnote{In \cite{BLN_Selecta}, the strong-type estimate \eqref{e:d=1boundaryend} was observed to hold in the smaller range $\beta \leq \frac{4}{3}$.}. We also remark that one cannot hope for \eqref{e:d=1boundaryend_weak} with $\beta \leq 2$; as we have already pointed out, the failure of the strong-type estimate \eqref{e:d=1boundaryend} (i.e. the missing endpoint in \eqref{e:FSlower} when $d = 1$) was demonstrated in \cite{FLLS}. In fact, by using the fact that Kakeya sets in $\mathbb{R}^2$ with zero Lebesgue measure exist, one can show that even the weak-type estimate \eqref{e:d=1boundaryend_weak} fails when $\beta = 2$ (see \cite{BHLNS}).

As far as we are aware, there are no further results available in the literature regarding the case $r = \infty$ and we present our new results in this direction below in Section \ref{section:mainresults}.

\subsection{The pointwise convergence problem}

Carleson's pointwise convergence problem for the Schr\"odinger flow $e^{it\Delta}f$ is concerned with identifying as large as possible class of initial data $f$ for which 
\begin{equation} \label{e:Carleson}
\lim_{t \to 0} e^{it\Delta}f(x) = f(x)
\end{equation}
holds for almost all $x \in \mathbb{R}^d$ (with respect to Lebesgue measure). Typically this is formulated in terms of data in the inhomogenous Sobolev space $H^s(\mathbb{R}^d)$ and then one wishes to identify the minimal regularity $s \in \mathbb{R}$ which guarantees \eqref{e:Carleson} holds for all $f \in H^s(\mathbb{R}^d)$. The sharpness of $s = \frac{1}{4}$ as the regularity threshold when $d = 1$ goes back to \cite{Carleson, DK}, and remarkable recent breakthroughs in \cite{Bourgain, DuGuthLi, DuZhang} have identified the threshold regularity\footnote{More precisely, when $d \geq 2$, for almost everywhere pointwise convergence to hold for all $f \in H^s(\mathbb{R}^d)$ it is known that $s > \frac{d}{2(d + 1)}$ is sufficient and $s \geq \frac{d}{2(d + 1)}$ is necessary, for which the reader is also referred to a survey paper \cite{Pr20} by Pierce, but the critical case remains open. When $d = 1$, it is a classical result that $s \geq \frac{1}{4}$ is necessary and sufficient.} to be $\frac{d}{2(d+1)}$ for all $d \geq 2$.

The problem described above has a natural analogue in the context of solutions to the reduced Hartree--Fock equation \eqref{e:vonNeumann}.  The interaction-free version of this equation \eqref{e:vonNeumann}, which we shall refer to as the von Neumann--Schr\"odinger equation (also known as the quantum Liouville equation) takes the form 
\begin{equation} \label{e:vonNeumannfree}
i \partial_t \gamma = [-\Delta,\gamma], \quad \gamma(0) = \gamma_0
\end{equation}
and, we recall, describes the time-evolution of a density operator. One may reconcile this with the Schr\"odinger flow $e^{it\Delta}f$ 
by identifying the initial data $f$ with the operator $\gamma_0 = \Pi_f$. Here we are assuming that $\|f \|_{L^2} = 1$ and $\Pi_f$ is the orthogonal projection operator onto the span of $f$ given by
\[
\Pi_f g = \langle g,f \rangle_{L^2} f.
\]
Indeed, in this case, one can easily verify that the solution of the von Neumann--Schr\"odinger equation is $\Pi_{e^{it\Delta}f}$. The above formulation in terms of the density operator provides a natural framework for modeling infinitely many (fermionic) particles.  

In general, the solution to \eqref{e:vonNeumannfree} is given by
\[
\gamma(t) = e^{it\Delta} \gamma_0 e^{-it\Delta}.
\]
Naturally associated with the flow $\gamma(t)$ is the density $\rho_{\gamma(t)}$ which, formally, is given by evaluating the integral kernel of the operator (a function on $\mathbb{R}^d \times \mathbb{R}^d$) on the diagonal. For example, we have $
\rho_{\gamma_0}(x) = \sum_{j = 1}^N \lambda_j |f_j(x)|^2
$
and
$
\rho_{\gamma(t)}(x) = \sum_{j = 1}^N \lambda_j |e^{it\Delta}f_j(x)|^2
$
if $\gamma_0$ is the finite-rank operator associated with orthonormal functions $f_1,\ldots,f_N \in L^2(\mathbb{R}^d)$ and scalars $\lambda_1,\ldots,\lambda_N$ given by
\[
\gamma_0f(x) = \sum_{j = 1}^N \lambda_j \Pi_{f_j}f(x) = \int_{\mathbb{R}^d} f(y) \sum_{j = 1}^N \lambda_j f_j(x) \overline{f_j(y)} \, \mathrm{d}y.
\]
With some care, one may extend the meaning of these densities to the infinite-rank case and beyond the trace class of operators; see 
the beginning of Section \ref{s:PW}. Associated with such densities, a natural analogue of Carleson's problem for the von Neumann--Schr\"odinger equation was raised in \cite{BLN_Selecta} and is concerned with finding as large as possible class of initial data $\gamma_0$ such that
\begin{equation} \label{e:aepointwise}
\lim_{t \to 0} \rho_{\gamma(t)}(x) = \rho_{\gamma_0}(x)
\end{equation}
holds almost everywhere with respect to Lebesgue measure. At the critical regularity $s = \frac{1}{4}$ in one spatial dimension, the following was proved in \cite{BLN_Selecta}.
\begin{theorem}[\cite{BLN_Selecta}] \label{t:pointwiseBLN}
The weak-type maximal-in-time estimate
\begin{equation} \label{e:ONmaximalweak}
\bigg\| \sum_j \lambda_j |Uf_j|^2 \bigg\|_{L^{2,\infty}_xL^\infty_t(\mathbb{R} \times \mathbb{R})} \lesssim \| \lambda \|_{\ell^\beta}
\end{equation}
holds for orthonormal functions $(f_j)_j$ in $\dot{H}^{\frac{1}{4}}(\mathbb{R})$ and $\beta < 2$. Consequently, if $\gamma_0 \in  \mathcal{C}^{\beta,\frac{1}{4}}$ and $\beta < 2$, then the density function $\rho_{\gamma(t)}$ satisfies \eqref{e:aepointwise} for almost every $x \in \mathbb{R}$.
\end{theorem} 
In the above statement, $\mathcal{C}^{\beta,\frac{1}{4}}$ denotes a Sobolev-type Schatten space. More generally, $\C^{\beta,s}$ is given by the norm\footnote{Although Carleson's problem is typically considered with initial data in the inhomogeneous Sobolev spaces $H^s(\mathbb{R}^d)$, as in Theorem \ref{t:pointwiseBLN}, we shall work in the setting of homogeneous Sobolev spaces $\dot{H}^s(\mathbb{R}^d)$ and their associated Schatten spaces $\C^{\beta,s}$. Our arguments may be easily be modified to the inhomogenous setting.}
\[
\|\gamma\|_{\C^{\beta,s}} = \||D|^s \gamma |D|^s \|_{\C^{\beta}},
\]
where $|D|^s$ is the Fourier multiplier operator with multiplier $|\xi|^s$ and $\mathcal{C}^\beta$ is the Schatten space of order $\beta$ built over $L^2(\mathbb{R}^d)$; we refer the reader forward to Sections~\ref{s:Motivation Outline} and \ref{s:PW} for more precise definitions. For now, we remark that the Schatten spaces $\mathcal{C}^\beta$ are nested (they increase in size as $\beta$ increases) and thus, for a fixed level of regularity $s$, it is natural to try and identify the largest possible $\beta$ for which \eqref{e:aepointwise} holds. Although we are unaware of a proof, for $s = \frac{1}{4}$ it seems reasonable to believe that $\beta < 2$ is the optimal range. Indeed, the maximal estimate \eqref{e:ONmaximalweak} was shown to fail for $\beta \geq 2$ in \cite{BLN_Selecta}.

There are many directions to develop Theorem \ref{t:pointwiseBLN}, some of which have already been raised in \cite{BLN_Selecta} and these open problems provided an important source of inspiration for the present paper. 
In particular, one may ask about: (a) \emph{the effect of imposing higher regularity on the initial data},  (b) \emph{extending to higher dimensions}, (c) \emph{generalizations to other equations}, (d) \emph{how large is the set of points at which convergence fails}.

Regarding (a), one may expect a gain in the range of allowable $\beta$ if the initial data is assumed to be sufficiently smooth. When $d = 1$ and $\gamma_0 \in \mathcal{C}^{\beta,s}$ with $s \in (\frac{1}{4},\frac{1}{2})$, an application of our boundary Strichartz estimates will reveal that $\beta < \frac{1}{1-2s}$ is allowable; this argument is in the spirit of \cite{BLN_Selecta} and relies heavily on the assumption the spatial dimension is one (we elaborate on this in Section \ref{subsection:Strichartz}). Making significant progress on (a)--(d) beyond this seems to require a different, more direct, approach. Even if we remain in the case $d = 1$, we can illustrate this by considering Carleson's problem
\begin{equation}\label{e:pointwise fractional}
    \lim_{t\to0}e^{it(-\Delta)^{m/2}}f(x)
    =
    f(x)\quad\ae
\end{equation}
for the fractional Schr\"{o}dinger propagators $e^{it(-\Delta)^{m/2}}$ with $m\in(0,\infty)$. For this, it is known that the threshold is remarkably different when $m\in(0,1)$ and $m \in (1,\infty)$\footnote{The case $m = 1$ also has a different nature; see the remark at the end of Section \ref{subsection:Strichartz} for further details.}. Indeed, when $m \in (0,1)$, Walther \cite{Wl97} proved $s>\frac m4$ is sufficient for \eqref{e:pointwise fractional} and $s\geq\frac m4$ is necessary. On the other hand, for $m \in (1,\infty)$, Sj\"{o}lin \cite{Sj87} has shown that \eqref{e:pointwise fractional} holds if and only if $s\geq\frac14$. Since Strichartz estimates are ultimately built on the standard dispersive estimates for $e^{it(-\Delta)^{m/2}}$, and since these dispersive estimates do not depend on $m$, it seems necessary to adopt a more direct approach to obtain \eqref{e:pointwise fractional}, or more generally, to satisfactorily address problem (c) above.  

Let us also mention that \eqref{e:pointwise fractional} has also been (partially) addressed in higher dimensions and in terms of the divergence sets. For example, it is known that \eqref{e:pointwise fractional} holds if $s > \frac m2$ for $m \in (0,1)$ and $d \geq 1$ (see \cite{Cowling, Wl97}). More recently, by building on \cite{DuGuthLi, DuZhang}, Cho--Ko \cite{CK18} showed that $s>\frac{d}{2(d+1)}$ is sufficient for \eqref{e:pointwise fractional} when $m \in (1,\infty)$ and $d \geq 2$. Regarding divergence sets of the form
\[
\mathfrak{D}(f) := \{ x \in \mathbb{R}^d : \lim_{t \to 0} e^{it(-\Delta)^{m/2}}f(x) \neq f(x)\},
\]
the idea of estimating their size seems to stem from work of Sj\"{o}gren--Sj\"{o}lin \cite{SS89}, with fresh impetus coming from the more recent paper by Barcel\'o \textit{et al.} \cite{BBCR}. Amongst other results, Barcel\'o \emph{et al}. proved that\footnote{Here, $\dim_H$ denotes Hausdorff dimension.} 
\[
\sup_{f\in H^s}\dim_H
\mathfrak{D}(f)
=
d-2s
\]
when $s\in[\frac{d}{4},\frac{d}{2})$ if either $d = 1$ and $m > 1$, or $d \geq 2$ and $m = 2$. The argument in \cite[Proposition 3.1]{BBCR} for $d \geq 2$ makes special use of the assumption $m = 2$ in order to reduce matters to one-dimensional considerations. We note that a consequence of our approach in this paper is that we can extend the $d \geq 2$ result to general $m \in (1,\infty)$. We also remark that whilst the result in \cite{BBCR} is  definitive when $d = 1$, the higher dimensional problem appears to be very challenging and remains open for $s \in (\frac{d}{2(d + 1)},\frac{d}{4})$ (see for example \cite{DuZhang, EP22, LR19a,LR19b} for further details).

In the case of $m\in(0,1)$ and $d = 1$, Cho and the third author \cite{CS22} very recently proved
\[
\sup_{f\in H^s}\dim_H
\mathfrak{D}(f)
\leq
\max\left\{
1-2s,
\frac12+\frac{1-4s}{2(1-m)}
\right\},
\]
and it seems reasonable to believe that equality holds. Including \cite{CS22}, ideas from the literature on the classical (single function) version of Carleson's problem have also been a source of inspiration for the present work. In particular, we deviate significantly from \cite{BLN_Selecta} and approach (a)--(d) in the context of density functions \eqref{e:aepointwise} using direct arguments for proving maximal-in-time estimates. Our new results in this direction appear in Section \ref{subsection:maximalintime}.

\section{Main new results} \label{section:mainresults}
\subsection{Maximal-in-space estimates} \label{subsection:Strichartz}
Our main result concerning boundary Strichartz estimates is the following. 
\begin{theorem} \label{t:maximalspace}
The estimate 
\begin{equation} \label{e:maximalspace} 
\bigg\| \sum_j \lambda_j |Uf_j|^2 \bigg\|_{L^{\frac{q}{2}}_tL^\infty_x(\mathbb{R} \times \mathbb{R}^d)} \lesssim \| \lambda \|_{\ell^\beta}
\end{equation}
holds for systems of orthonormal functions $(f_j)_j$ in $\dot{H}^{s}(\mathbb{R}^d)$, $s = \frac{d}{2} - \frac{2}{q}$, and $\beta < \frac{q}{2}$ in each of the following cases.

\emph{(i)} $d = 1$ and $q \in (4,\infty)$. 

\emph{(ii)} $d \geq 2$ and $q \in (2,\infty)$.

\noindent Furthermore, the estimate \eqref{e:maximalspace} fails when $\beta > \frac{q}{2}$.
\end{theorem}
 For $d \geq 2$, thanks to the failure of \eqref{e:classicalSfalse_allq} in both cases $q = 2, \infty$, the range of $q$ in (ii) cannot be extended. It remains an interesting open problem to determine whether one can extend the range of $q$ in (i) to $[4,\infty)$. 
 
 Regarding the summability exponent, the failure of \eqref{e:maximalspace} when $\beta > \frac{q}{2}$ can be seen by following the argument in \cite[Section 4.8]{BHLNS} and so we omit the details. What happens at the critical summability exponent $\beta = \frac{q}{2}$ seems to be a delicate matter. In Section \ref{subsection:motivation} we shall sketch an argument yielding restricted weak-type estimates when $\beta = \frac{q}{2}$ and $d \geq 2$. Interestingly, such estimates are not valid when $d = 1$. Indeed, if we take $q \in [4,\infty)$ and assume that
 \begin{equation} \label{est:sharpnessTheorem1.2}
\bigg\| \sum_j \lambda_j |Uf_j|^2 \bigg\|_{L^{\frac{q}{2},\infty}_tL^{\infty}_x(\R \times \R)} \lesssim \| \lambda \|_{\ell^{\frac{q}{2},1}}
\end{equation}
were to hold for all orthonormal systems $(f_j)_j$ in $\dot{H}^{\frac12 - \frac2q}(\R)$, then by a semi-classical limiting argument (see \cite{Sabin_survey, BHLNS}), we get the (``velocity average") estimate
\[
\biggl\| \int_{\R} f(x-tv, v) \frac{\d v}{|v|^{1- \frac{4}{q}}} \biggr\|_{L_t^{\frac{q}{2},\infty} L_x^{\infty}(\R \times \R)} \lesssim 
\|f \|_{L^{\frac{q}{2},1}}.
\]
However, by taking $f$ as the characteristic function of a sufficiently small neighbourhood of a measure-zero Kakeya set in $\R^2$, we can show that such an estimate cannot be true (see the proof of \cite[Theorem 5.3]{BHLNS}). We leave open the question of whether one can extend Theorem \ref{t:maximalspace} to the critical summability exponent $\beta = \frac{q}{2}$ when $d \geq 2$; conceivably, in the scale of Lorentz spaces, restricted weak-type estimates are the best that one can hope for.

Our proof of Theorem \ref{t:maximalspace} draws on ideas in \cite{BHLNS} and relies on widening the framework to incorporate Lorentz spaces. Indeed, restricted weak-type versions of \eqref{e:ONSqr} can be readily obtained by summing up certain frequency-localized estimates (which in turn are based on dispersive estimates) using the summation trick of Bourgain in Lemma \ref{l:Bourgain's trick} below. Also, in certain special cases, improved versions of \eqref{e:ONSqr} in the scale of Lorentz spaces are possible and offsetting this loss/gain, via interpolation, led to the desired strong-type estimates in \cite{BHLNS}. Broadly speaking we adopt this strategy to prove Theorem \ref{t:maximalspace} and we give a more informative outline of the proof in Section \ref{subsection:Outline}.

Theorem \ref{t:maximalspace} readily generalises to fractional Schr\"odinger propagators $U_m$, with $m \in (0,\infty) \setminus \{1\}$, where
\[
U_m f(t,x) = e^{it(-\Delta)^{m/2}}f(x) = \frac{1}{(2\pi)^d}\int_{\mathbb{R}^d} e^{i(x \cdot \xi + t|\xi|^m)} \widehat{f}(\xi) \, \mathrm{d}\xi.
\]
The propagators $U_m$ have drawn much attention from both physical and mathematical viewpoints in recent years (see for example \cite{IP, Karpman1, Karpman2, KPV, Laskin1, Laskin2}). As described above, the proof of Theorem \ref{t:maximalspace} rests on a dispersive estimate which is equally valid for general $U_m$ and it is easily checked that our proof in Sections \ref{s:Upper half} and \ref{s:lower half} goes through for $U_m$, as long as we consider orthonormal systems of data $(f_j)_j$ belonging to $\dot{H}^{s}(\mathbb{R}^d)$, $s = \frac{d}{2} - \frac{m}{q}$. In particular, the range of $d, q$ and $\beta$ is unchanged from the case $m = 2$ in Theorem \ref{t:maximalspace} above. 

One advantage of broadening the framework to general $U_m$ can be illustrated as follows. In the case of $d = 1$, a simple change of variables shows\footnote{Here, $A \simeq B$ means $A = cB$ for some positive constant $c$ whose exact value is not important.}
\[
U_mf(t,x) \simeq U_{1/m} f_+(x,t) + U_{1/m} f_-(-x,t), \qquad \widehat{f_\pm}(\xi) = \chi_{(0,\infty)}(\xi)\widehat{f}(\pm \xi^\frac{1}{m})\xi^{\frac{1}{m} - 1}
\]
which means that, essentially, one can swap the roles of space and time at the expense of switching from $U_m$ to $U_{1/m}$ (this trick has been used many times and goes back to Kenig--Ponce--Vega \cite{KPV}). Modulo an additional argument which takes care of the orthogonality condition (see \cite[Lemma 4.2]{BLN_Selecta}), this line of reasoning allows us to deduce the maximal-in-time estimate
\begin{equation} \label{e:ONmaximal}
\bigg\| \sum_j \lambda_j |U_mf_j|^2 \bigg\|_{L^{\frac{1}{1-2s}}_xL^\infty_t(\mathbb{R} \times \mathbb{R})} \lesssim \| \lambda \|_{\ell^\beta}
\end{equation}
for orthonormal functions $(f_j)_j$ in $\dot{H}^{s}(\mathbb{R})$, whenever $s \in (\frac{1}{4},\frac{1}{2})$ and\footnote{This range of $\beta$ is best possible for the estimate \eqref{e:ONmaximal} to hold. In fact, for $s \in (\frac{1}{4},\frac{1}{2})$, the failure of the following maximal-in-time estimate
\[
\bigg\| \sum_j \lambda_j |Uf_j|^2 \bigg\|_{L^{\frac{1}{1-2s},\infty}_xL^\infty_t(\mathbb{R} \times \mathbb{R})} \lesssim \| \lambda \|_{\ell^{\frac{1}{1-2s},1}}
\]
for orthornormal functions $(f_j)_j$ in $\dot{H}^s(\R)$ can be shown by adapting the argument discussed earlier for \eqref{est:sharpnessTheorem1.2}. We refer the reader to \cite{BLN_Selecta} for the details, where such an argument was given in the case $s = \frac{1}{4}$.} $\beta < \frac{1}{1-2s}$. Here we can take any $m \in (0,\infty) \setminus \{1\}$ but we note that the $(s,\beta)$ range is independent of $m$. As a result of the maximal estimate \eqref{e:ONmaximal}, it follows that if $\gamma_0 \in  \mathcal{C}^{\beta,s}$, then the density functions $\rho_{\gamma(t)}$ and $\rho_{\gamma_0}$ satisfy \eqref{e:aepointwise} for almost every $x \in \mathbb{R}$. Here
\[
\gamma(t) = e^{it(-\Delta)^{m/2}}\gamma_0e^{-it(-\Delta)^{m/2}}
\]
is the solution to fractional von Neumann--Schr\"odinger equation
\begin{equation} \label{e:fractionalvN-S} 
i\partial_t \gamma = -[(-\Delta)^{\frac{m}{2}},\gamma], \quad \gamma(0) = \gamma_0.
\end{equation}

As mentioned earlier, some of our original motivation to prove boundary Strichartz estimates arose from the problem of understanding the pointwise convergence \eqref{e:aepointwise} in the context of imposing higher regularity on the data, and thus address a question raised in \cite{BLN_Selecta}. However, if one is solely focused developing a better understanding of the pointwise convergence problem \eqref{e:aepointwise}, then it seems strongly desirable to have a more direct approach for proving maximal-in-time estimates. For a start, one can then make the problem more tractable by targeting local estimates (with respect to both time and space) rather than global estimates like those in \eqref{e:ONmaximal}. As we have already mentioned\footnote{For example, we refer the reader back to the discussion following Theorem \ref{t:pointwiseBLN}.}, obtaining a new viewpoint on \eqref{e:aepointwise} was another major source of motivation for the present paper and next we shall present our new results in this direction. Let us remark that by succeeding in this manner, we will profit in two different ways. Firstly, we are able to obtain, for the first time, estimates on the size of the set of points at which the convergence in \eqref{e:aepointwise} fails. Secondly, we are able to obtain what we believe to be sharp results for $U_m$ when $m \in (0,1)$; for such $m$, the $(s,\beta)$ range depends on $m$ and obtaining such an outcome by relying on Strichartz estimates (and the space-time switching trick) seems rather implausible.

\begin{remark}
For the one-sided wave propagator $U_1$, the classical Strichartz estimates and Carleson's problem have also been addressed. For Strichartz estimates, roughly speaking, the admissible exponents for $U_1$ on $\mathbb{R} \times \mathbb{R}^d$ correspond to those for $U$ on  $\mathbb{R} \times \mathbb{R}^{d-1}$, and to a large extent a unified perspective is possible. In the same spirit, an analogue of Theorem \ref{t:maximalspace} holds for $U_1$; using ideas from the present paper, a sketch of the proof of this result can be found in \cite{BKS_RIMS}. 

On the other hand, Carleson's problem for $U_1$ is significantly different to the problem for $U$. Whereas $\frac{d}{2(d + 1)}$ has been identified as the sharp regularity threshold for \eqref{e:Carleson}, it is known (and can be proved using significantly easier arguments) that $\frac{1}{2}$ is the sharp regularity threshold for $U_1$ in all dimensions (see \cite{Cowling, Wl97}). In the present paper, we do not attempt to consider the analogue of \eqref{e:aepointwise} for $U_1$.
\end{remark}

\subsection{Maximal-in-time estimates and pointwise convergence} \label{subsection:maximalintime}

Consider the divergence set 
\[
\mathfrak{D}(\gamma_0)
:=
\{x\in \mathbb R^d:\lim_{t\to 0}\rho_{\gamma(t)}(x)\not=\rho_{\gamma_0}(x)\}
\]
associated with the initial data $\gamma_0$, and where $\gamma(t)$ is the solution of \eqref{e:fractionalvN-S}. For ease of notation, we are suppressing the dependence on $d$ and $m$. Any upper bound on the Hausdorff dimension of $\mathfrak{D}(\gamma_0)$ which is strictly less than $d$ would certainly be a stronger statement than the almost everywhere pointwise convergence of $\rho_{\gamma(t)}$ to $\rho_{\gamma_0}$ and, in Corollary \ref{c:div set} below, we provide upper bounds of this nature. By Frostman's lemma from geometric measure theory, it will suffice to establish local maximal-in-time estimates where integration in the spatial variable is taken with respect to an arbitrary $\alpha$-dimensional measure. Here, for a given $\alpha\in(0,d]$, the Borel measure $\mu$ on $\R^d$ is said to be $\alpha$-dimensional if 
\[
\sup_{x\in\R^d, r>0}\frac{\mu(B(x,r))}{r^\alpha}<\infty,
\]
and we shall use $\mathcal M^\alpha(\mathbb B^d)$ to denote the collection of all $\alpha$-dimensional probability measures supported on
the unit ball $\mathbb B^d$. Our maximal-in-time estimates appear in the forthcoming Theorem \ref{t:maximal main} and take the form
\begin{equation}\label{maximal}
            \bigg\|\sum_j\lambda_j|U_m f_j|^2\bigg\|_{L_x^1(\mathbb B^d,\d\mu)L_t^\infty(0,1)}
            \lesssim
            \|\lambda\|_{\ell^\beta}.
        \end{equation}
Together with Corollary \ref{c:div set}, these results provide a significant extension and refinement of Theorem \ref{t:pointwiseBLN}.
\begin{theorem}\label{t:maximal main}
    Let $d\geq1$, $s \in [0,\frac{d}{2})$, $\alpha \in (0,d]$, $\beta \geq 1$, and $\mu\in \mathcal M^\alpha(\mathbb B^d)$.
    \begin{enumerate}[(i)]
        \item
        \label{item:t:maximal m larger 1}Let $m\in(1,\infty)$ and $s \geq \frac{d}{4}$. The maximal estimate \eqref{maximal} holds for all orthonormal systems $(f_j)_j$ in $\dot{H}^s({\mathbb R}^d)$ whenever 
        \[
        s > 
        \frac12\left(d-\frac{\alpha}{\beta}\right).
        \]
        This is sharp in the sense that if $s < \frac12(d-\frac{\alpha}{\beta})$, then there exist $\mu\in \mathcal M^\alpha(\mathbb{B}^d)$, an orthonormal system $(f_j)_j$ in $\dot{H}^s({\mathbb R}^d)$, and a sequence $(\lambda_j)_j \in \ell^\beta$ such that \eqref{maximal} fails.
        \item
        \label{item:t:maximal m less 1} Let $m\in (0,1)$. The maximal estimate
        \eqref{maximal} 
        holds for all orthonormal systems $(f_j)_j$ in $\dot{H}^{s}({\mathbb R}^d)$ whenever 
        \[
        s > \max\left\{
        \frac12\left(d-\frac{\alpha}{\beta}\right),
        \frac d4-\frac12(1-m)\left(\frac\alpha\beta-\frac d2\right) 
        \right\}.
        \]
        When $d=1$, this is sharp in the sense that if
        $\max\left\{\frac12(1-\frac{\alpha}{\beta}), \frac14-\frac12(1-m)(\frac\alpha\beta-\frac12)\right\}>s$, then there exist $\mu\in \mathcal M^\alpha(\mathbb{B})$, an orthonormal system $(f_j)_j$ in $\dot{H}^s({\mathbb R})$, and a sequence $(\lambda_j)_j\in \ell^\beta$  such that \eqref{maximal} fails.
    \end{enumerate}
\end{theorem}
Taking $\mu$ as standard Lebesgue measure (and $\alpha = d$), one may deduce the pointwise convergence
\[
\lim_{t\to0} \rho_{\gamma(t)}(x)
=
\rho_{\gamma_0}(x)\quad \ae
\]
for self-adjoint initial data $\gamma_0\in \C^{\beta,s}$, with the appropriate $(\beta,s)$ given in the above theorem. In particular, we note that the range of $s$ may drop below $\frac{d}{4}$ in the case $m \in (0,1)$. We also note that the full strength of Theorem \ref{t:maximal main} yields geometric size information on the corresponding divergence set as follows. We illustrate this in Figure \ref{f:divset} in the case $d = 1$.

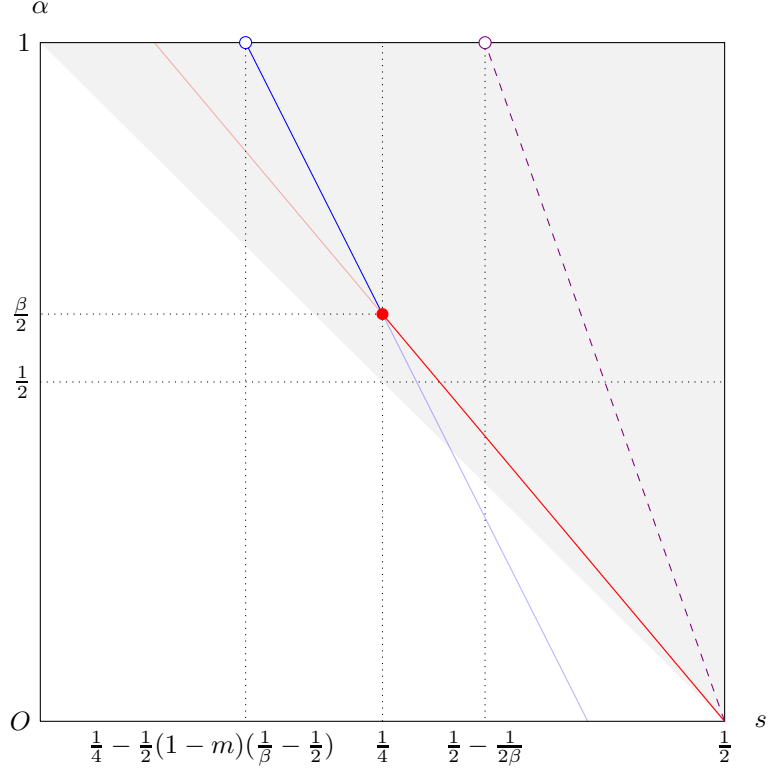
\begin{figure}[t]
\begin{center}
\begin{tikzpicture}[scale=9]
\coordinate (O) at (0,0) node [left] {$O$};
\fill [gray,opacity=0.1] (1,0)--(1,1)--(0,1);
\draw (0,0)--(1,0)--(1,1)--(0,1)--(0,0);
\node [right] at (1.03,0) {$s$};
\node [above] at (0,1.03) {$\alpha$};

\draw [dotted](1/2,0)--(1/2,1);
\draw [dotted](0,1/2)--(1,1/2);

\draw [red] (1,0)--(1/2,1/2+0.1);
\draw [blue] (1/2,1/2+0.1)--(1/2-0.2,1);
\draw [red, opacity=0.3] (1,0)--(1/6,1);
\draw [blue,opacity=0.3] (1/2,1/2+0.1)--(1/2+0.3,0);

\draw [dotted] (1/2,1/2+0.1)--(0,1/2+0.1) node [left] {$\frac\beta2$};
\draw [dotted] (1/2-0.2,1)--(1/2-0.2,0);
\node [below] at (1/2-0.25,0) {$\frac14-\frac12(1-m)(\frac1\beta-\frac12)$};
\node [below] at (1/2,0) {$\frac14$};
\node [below] at (1,0) {$\frac12$};
\node [left] at (0,1/2) {$\frac12$};
\node [left] at (0,1) {$1$};

\fill [white] (1/2-0.2,1) circle (0.25pt);
\draw [blue] (1/2-0.2,1) circle (0.25pt);
\fill [red] (1/2,1/2+0.1) circle (0.25pt);

\coordinate (B) at (0.65,1);
\draw [dotted] (B)--(B|-O);
\draw [violet,dashed] (1,0)--(B);
\fill [white] (B) circle (0.25pt);
\draw [violet] (B) circle (0.25pt);

\node [below] at (B|-O) {$\frac12-\frac{1}{2\beta}$};






\end{tikzpicture}
\caption{The graph illustrates our upper bound of the Hausdorff dimension of the divergence sets when $d=1$. For $m\in(1,\infty)$, it is given by \textcolor{red}{$\alpha = (1-2s)\beta$}. This is extended to \textcolor{blue}{$\alpha = \frac{1-2s-\frac m2}{1-m}\beta$} up to $s=\frac14-\frac12(1-m)(\frac1\beta-\frac12)$ when $m\in(0,1)$. For $\beta\in[2,\infty)$, the situation changes and $\alpha = (1-2s)\beta$ is the \textcolor{violet}{dashed line}.}\label{f:divset}
\end{center}

\end{figure}

\begin{corollary}\label{c:div set}
Let $d \geq1$, $s \in [0,\frac{d}{2})$. Suppose $\gamma_0 \in \C^{\beta,s}$ is self-adjoint with $1 \leq \beta<\frac{d}{d-2s}$. 
\begin{enumerate}[(i)]
\item
For $m\in(1,\infty)$ and $s \geq \frac d4$, the pointwise convergence \eqref{e:aepointwise} holds and furthermore we have
\[
\dim_H\mathfrak D(\gamma_0)
\leq
(d-2s)\beta.
\]
\item \label{item:divset m less 1}
For $m\in(0,1)$ and $s > \frac d4-\frac d2(1-m)(\frac 1\beta-\frac12)$, the pointwise convergence \eqref{e:aepointwise} holds and furthermore we have
\[
\dim_{H}\mathfrak D(\gamma_0)
\leq
\max\left\{
(d-2s)\beta,
\frac{2(d-2s)-md}{2(1-m)}\beta
\right\}.
\]
\end{enumerate} 
\end{corollary}
We reiterate that, even when $\beta = 1$, Corollary \ref{c:div set}(i) appears to be new for $m \neq 2$ when $d \geq 2$. Also, the case $\beta = 1$ in Corollary \ref{c:div set}(ii) corresponds to \cite[Theorem 1]{CS22}.

In light of the necessary condition in Theorem \ref{t:maximal main}(i), it seems reasonable to believe that the upper bound in Corollary \ref{c:div set}(i) is best possible for any $d \geq 1$. Similarly, we expect the upper bound in Corollary \ref{c:div set}(ii) to be sharp when $d = 1$, but the sharp threshold on the size of the divergence sets when $m \in (0,1)$ and $d \geq 2$ is less clear.

\section{The critical case and outline of the proof of Theorem \ref{t:maximalspace}}\label{s:Motivation Outline}
The proof of Theorem \ref{t:maximalspace} in full appears in Sections \ref{s:Upper half} and \ref{s:lower half}, but later in this section we include an overview of the main steps. Prior to that, it will be informative to briefly sketch an argument which yields restricted weak-type estimates at the critical exponent $\beta = \frac{q}{2}$. The corresponding estimates for the one-sided wave propagator $U_1$ were carefully proved in \cite{BKS_RIMS} and so we refer the reader there for further details.

 \subsection{Restricted weak-type estimates at $\beta = \frac{q}{2}$} \label{subsection:motivation}
For $d\ge2$, we claim
\begin{equation}\label{i:weak-restricted}
\bigg\|\sum_j\lambda_j|Uf_j|^2\bigg\|_{L_t^{\frac q2,\infty}L_x^\infty}
\lesssim 
\|\lambda\|_{\ell^{\frac q2,1}}
\end{equation}
holds for all orthonormal systems $(f_j)_j$ in $\dot{H}^s(\mathbb R^d)$, $s = \frac{d}{2} - \frac{2}{q}$, and all sequences $(\lambda_j)_j$ in $\ell^{\frac q2,1}$, where the range of $q$ is the same as in Theorem \ref{t:maximalspace}. 

First consider $d \geq 3$. Our argument is based on the fact that if $q$ coincides either with $2$ or $\infty$, then the frequency-localized estimate
\begin{equation} \label{e:ONSqrfreqlocal}
\bigg\| \sum_j \lambda_j |UP_kf_j|^2 \bigg\|_{L^{\frac{q}{2}}_tL^\infty_x} \lesssim 2^{(\frac{d}{2} - \frac{2}{q})k}\| \lambda \|_{\ell^\frac{q}{2}}
\end{equation}
holds for all orthonormal systems $(f_j)_j$ in $L^2(\mathbb R^d)$. Here, $k$ is an integer and the operator $P_k$ denotes the Littlewood--Paley type frequency localization given by $\widehat{P_k f}(\xi)=\varphi_k(|\xi|)\widehat{f}(\xi)$, where $\varphi_k(r)=\varphi(2^{-k}r)$ with $\varphi\in C_0^\infty$ supported in $\{r \in\mathbb R:2^{-1}< r <2\}$. The estimate \eqref{e:ONSqrfreqlocal} with $q = 2$ follows from the frequency-localized Strichartz estimate
\begin{equation} \label{e:classicaltopleft}
\|UP_kf\|_{L^2_tL^\infty_x(\mathbb{R} \times \mathbb{R}^d)} \lesssim 2^{\frac{d-2}{2}k} \|f\|_{L^2(\mathbb{R}^d)}
\end{equation}
and the triangle inequality, and \eqref{e:ONSqrfreqlocal} with $q = \infty$ follows from Bessel's inequality. The latter claim is based on the following simple observation from \cite{BHLNS}:
\[
\sum_j |UP_kf_j(t,x)|^2 =  \sum_j |\langle \widehat{f_j}(\xi), e^{-ix \cdot \xi} m_t(\xi) \rangle_{L^2_\xi}|^2 \leq \|m_t\|_2^2 \lesssim 2^{kd}. 
\]
Here, $m_t$ is the Fourier multiplier  associated with $UP_kf(t,\cdot)$ and we have used the fact that, thanks to Parseval's identity, $(\widehat{f_j})_j$ is orthonormal in $L^2(\mathbb{R}^d)$.

Finally, in order to sum up the frequency-localized estimates and pass to \eqref{i:weak-restricted}, we use the following. 
\begin{lemma}[\cite{BHLNS, BKS_RIMS}]\label{l:Bourgain's trick}
Let $q_0, q_1,r \in [2,\infty]$, $\beta_0, \beta_1 \in [2,\infty]$ and $(g_j)_j$ be a uniformly bounded sequence in $L_t^{q_0}L_x^{r}\cap L_t^{q_1}L_x^{r}$. Suppose there exist $\varepsilon_0$, $\varepsilon_1>0$ such that 
\[
\bigg\|\sum_j\lambda_j|P_kg_j|^2\bigg\|_{L_t^{\frac{q_i}{2},\infty}L_x^{\frac{r}{2}}}\lesssim 2^{(-1)^{i+1}\varepsilon_ik}\|\lambda\|_{\ell^{\beta_i}}
\]
for all $k\in\mathbb{Z}$ and $i=0,1$, then
\[
\bigg\|\sum_j\lambda_j|g_j|^2\bigg\|_{L_t^{\frac{q}{2},\infty}L_x^{\frac{r}{2}}}\lesssim \|\lambda\|_{\ell^{\beta,1}},
\]
where $\frac{1}{q}=\frac{1-\theta}{q_0}+\frac{\theta}{q_1}$, $\frac{1}{\beta}=\frac{1-\theta}{\beta_0}+\frac{\theta}{\beta_1}$ and $\theta=\frac{\varepsilon_0}{\varepsilon_0+\varepsilon_1}$.
\end{lemma}
We are referring the reader to \cite{BHLNS, BKS_RIMS} for the above lemma since the statement can be found as it is written here in these papers, but the key underlying idea goes back to Bourgain \cite{Bourgaintrick}.

When $d = 2$ may be proved in a similar manner but the failure of \eqref{e:classicaltopleft} (and thus also \eqref{e:ONSqrfreqlocal} with $q = 2$) adds difficulty. Via Lemma \ref{l:Bourgain's trick} again, it actually suffices to establish \eqref{e:ONSqrfreqlocal} with $q$ strictly larger but arbitrarily close to $2$. One may reach such estimates by using duality (the forthcoming Lemma \ref{l:duality}) and estimates on certain Schatten norms using the Brascamp--Lieb inequality. We omit the details and refer the interested reader to the analogous argument for the half-wave propagator in \cite{BKS_RIMS}.

 \subsection{Outline of the proof of Theorem \ref{t:maximalspace}} \label{subsection:Outline}
It is unclear to us whether it is possible to upgrade \eqref{i:weak-restricted} to a strong-type estimate at the critical exponent $\beta = \frac{q}{2}$. In Theorem \ref{t:maximalspace} we obtain strong-type estimates for $\beta < \frac{q}{2}$ by dividing the range of $q$ into $2<q\leq4$ and $4 < q<\infty$ (of course, when $d=1$ this division is superfluous). In this overview of the proof of Theorem \ref{t:maximalspace}, for simplicity let us suppose $d \geq 2$.

As suggested by the above division, key to the argument is $q$ close to $4$. More precisely, the case $\beta = 2$ is important for our proof and we shall see that
\begin{equation}\label{i:St middle}
    \bigg\|\sum_j\lambda_j|U|D|^{-s}f_j|^2 \bigg\|_{L_t^{\frac q2,2}L_x^\infty}
    \lesssim
    \|\lambda\|_{\ell^2}
\end{equation}
for $4 < q < \infty$ (which is of most use to us when $q$ is close to $4$) and $s = \frac{d}{2} - \frac{2}{q}$. Here, $L_t^{\frac q2,2}$ is a Lorentz space. Since $q > 4$ these estimates are stronger than the strong-type counterpart with $L_t^{\frac q2}$, and this gain will be important to obtain Theorem \ref{t:maximalspace} when $4 < q < \infty$. 

We first consider $2<q\leq 4$. Here we shall see that \begin{equation}\label{i:St top}
    \bigg\|\sum_j\lambda_j|U|D|^{-s}f_j|^2 \bigg\|_{L_t^{\frac q2,1}L_x^\infty}
    \lesssim
    \|\lambda\|_{\ell^1}
\end{equation}
holds for any $q>2$ and $s = \frac{d}{2} - \frac{2}{q}$, and this is of most use to us when $q$ is close to $2$. We would like to use complex interpolation to deduce Theorem \ref{t:maximalspace}\footnote{In fact, a slightly stronger estimate involving certain Lorentz spaces} for $2 < q \leq 4$ from \eqref{i:St middle} and \eqref{i:St top} but the fact that $s$ varies with $q$ causes some technical difficulty. In order to carry out an analytic interpolation argument to take care of this, we found that it was more convenient to reformulate and work with dual estimates. Here, duality is in the sense of the forthcoming Lemma \ref{l:duality} and we describe the argument on the dual side in more detail in a moment.

Before that, let us explain the idea behind the proof for $4 < q < \infty$. We shall first obtain the weak-type estimate
\begin{equation}\label{i:St bottom}
    \bigg\|\sum_j\lambda_j|U|D|^{-s}f_j|^2 \bigg\|_{L_t^{\frac q2,\infty}L_x^\infty}
    \lesssim
    \|\lambda\|_{\C^{\beta}},
\end{equation}
for $4 < q < \infty$, $\beta<\frac{q}{2}$ (which is of most use to us when $q$ is extremely large) and $s = \frac{d}{2} - \frac{2}{q}$. Again one would like to interpolate this with \eqref{i:St middle} to obtain the desired estimates in Theorem \ref{t:maximalspace} for $4 < q < \infty$; in particular, it is key that the slight gain in the sense of the Lorentz space in \eqref{i:St middle} can be used to upgrade the weak-type estimate \eqref{i:St bottom} to a strong-type estimate.

In order to argue along the above lines, as we mentioned above, we found it more convenient to work on dual estimates using the following. 
\begin{lemma}[Duality principle \cite{FS_AJM}, \cite{BHLNS}]\label{l:duality}
Suppose $T$ is a bounded linear operator from $L^2(\mathbb{R}^d)$ to $L_t^{q,a} L_x^{r,b}$ for some $q,r,a,b \geq2$ and under the condition that $a = 2$ when $q = 2$. Also, let $\beta\geq1$.  Then, 
\[
\bigg\|\sum_j\lambda_j|Tf_j|^2 \bigg\|_{L_t^{\frac {q}{2},\frac {a}{2}}L_x^{\frac {r}{2},\frac{b}{2}}}\lesssim\|\lambda\|_{\ell^{\beta}}
\]
holds for all orthonormal systems $(f_j)_j$ in $L^2(\mathbb R^d)$ and all sequences $(\lambda_j)_j \in \ell^{\beta}(\mathbb C)$ if and only if 
\[
\|WTT^*\overline{W}\|_{\C^{\beta'}}
\lesssim 
\|W\|_{L_t^{\widetilde{q},\widetilde{a}}L_x^{\widetilde{r},\widetilde{b}}}^2
\]
for all $W\in L_t^{\widetilde{q},\widetilde{a}}L_x^{\widetilde{r},\widetilde{b}}$. Here, $\widetilde{\cdot}$ denotes the ``half conjugate" given by
\[
\frac{1}{p}+\frac{1}{\widetilde{p}}=\frac12
\]
and $\cdot'$ denotes the (usual H\"older) conjugate given by
\begin{align*}
\frac1\beta+\frac{1}{\beta'}=1.
\end{align*}
\end{lemma}
In the above lemma, for $p > 1$, the notation $\C^{p}$ denotes the Schatten space of all compact operators on $L^2(\mathbb R^{d+1})$ such that $\|T\|_{\C^{p}}:=\|\lambda\|_{\ell^{p}}<\infty$, where $(\lambda_j)_j$ are the singular values of $T$. We extend this to $p = \infty$ with the convention that $\|\cdot\|_{\C^\infty}$ is the usual operator norm. In our proofs, we often recall that $\|\cdot\|_{\C^2}$ is the Hilbert--Schmidt norm and coincides with the norm $\|\cdot\|_{L^2(\mathbb R^{d+1}\times\mathbb R^{d+1})}$ of the kernel of the corresponding operator. 
Moreover, we make use several times of a useful characterization of Schatten norms.
\begin{lemma}[{\cite[Proposition~2.6]{Simon}}]
\label{l:Simon}
Let $n \in \N$ and $\mathcal{B}$ denote the collection all orthonormal families in $L^2(\R^n)$. For $1 \leq p \leq \infty$ and $T \in \mathcal{C}^p(L^2(\R^n))$ we have
\begin{equation}\label{est:Prop2.6Simon-1}
\| T\|_{\mathcal{C}^{p}} = \sup_{\phi, \psi \in \mathcal{B}} \|\langle  T \phi_j,   \psi_j \rangle_{L^2} \|_{\ell^{p}}.
\end{equation}
Conversely, if $T$ is compact in $L^2(\R^n)$ and the right-hand side of \eqref{est:Prop2.6Simon-1} is finite, then $T \in \mathcal{C}^{p}(L^2(\R^n))$. When $1\leq p<\infty$, ``$ \, T$ is compact" may be replaced by ``T is bounded" in the last statement.
\end{lemma}

Thanks to Lemma \ref{l:duality}, our desired estimates follow from
\begin{equation}\label{i:main dual}
    \|WU|D|^{-2s}U^*\overline{W}\|_{\C^{\beta'}}
    \lesssim
    \|W\|_{L_t^{\widetilde{q}}L_x^2}^2
\end{equation}
for $\beta'>\frac {\widetilde{q}}{2}$ and $2 < \widetilde{q} < \infty$, where
\[
s=\frac{d}{2}-\frac 2q=\frac{d-2}{2}+\frac{2}{\widetilde{q}}.
\]
To obtain such estimates via an analytic interpolation argument, we prove the following slightly strenghened versions of \eqref{i:St middle}, \eqref{i:St top} and \eqref{i:St bottom}:
\begin{equation}\label{i:C middle}
    \|W_1U|D|^{-2s+i\kappa}U^*W_2\|_{\C^2}
    \leq C(\kappa)
    \|W_1\|_{L_t^{\widetilde{q},4}L_x^2}\|W_2\|_{L_t^{\widetilde{q},4}L_x^2}
\end{equation}
for $2 < \widetilde{q} < 4$,
\begin{equation}\label{i:C top}
    \|W_1U|D|^{-2s+i\kappa}U^*W_2\|_{\C^\infty}
    \leq C(\kappa)
    \|W_1\|_{L_t^{\widetilde{q},\infty}L_x^2}\|W_2\|_{L_t^{\widetilde{q},\infty}L_x^2}
\end{equation}
for $2 < \widetilde{q} < \infty$, and 
\begin{equation}\label{i:C bottom}
    \|W_1U|D|^{-2s+i\kappa}U^*W_2\|_{\C^{\beta'}}
    \leq C(\kappa)
    \|W_1\|_{L_t^{\widetilde{q},2}L_x^2}\|W_2\|_{L_t^{\widetilde{q},2}L_x^2}
\end{equation}
for $\beta'>\frac{\widetilde{q}}{2}$, $2<\widetilde{q}<4$. In each case, $\kappa\in\mathbb R$ and $C(\kappa)$ is a constant which grows subexponentially with $\kappa$. As we shall see, the estimates \eqref{i:C middle} and \eqref{i:C top} follow from a certain dispersive estimate and O'Neil's refined version of Hardy--Littlewood--Sobolev inequality, whilst \eqref{i:C bottom} follows from appropriate frequency-local estimates and use of bilinear interpolation (in the spirit of the Keel--Tao argument \cite{KeelTao}) to sum up the localized estimates.

\section{Proof of Theorem \ref{t:maximalspace} for $2 < q \leq 4$}\label{s:Upper half}

First we claim that it suffices to prove \eqref{e:maximalspace} with $U$ replaced by $U_{\leq 1} := U \chi(|D|)$, where $\chi \in C_0^\infty(-4,4)$. Indeed, from the global nature of the estimate, a simple rescaling argument would then give \eqref{e:maximalspace} with $U \chi(\varepsilon|D|)$ for any $\varepsilon > 0$ and the desired estimate follows by taking $\varepsilon$ to zero.

When $2 < q \leq 4$, even though the claimed estimates are only valid for $d \geq 2$, some of the preparatory results are valid for $d = 1$ too so we include these cases where possible. Also, it will be handy to introduce the notation $A \lessapprox_\kappa B$ to mean $A \leq C(\kappa) B$, where $C(\kappa) \lesssim e^{\varepsilon |\kappa|}$ for any $\varepsilon >0$.

\subsection{Preparation for the interpolation}

\begin{proposition}
\label{proposition:UpperHalfC2Cinfty}
Let $d \geq 1$. If $4<q_0<\infty$, then
\begin{equation}\label{i:C2}
\|W_1U_{\leq 1}|D|^{-2s_0+i\kappa}U_{\leq 1}^*W_2\|_{\C^2}
\lessapprox_\kappa 
\|W_1\|_{L_t^{\widetilde{q_0},4}L_x^2}\|W_2\|_{L_t^{\widetilde{q_0},4}L_x^2}
\end{equation}
where $s_0 = \frac{d}{2}- \frac{2}{q_0}$. Also, if either $d \geq 2$ and $2 < q_1 < \infty$, or $d = 1$ and $4 \leq q_1 < \infty$, then
\begin{equation}\label{i:Cinfty}
\|W_1U_{\leq 1}|D|^{-2s_1+i\kappa}U_{\leq 1}^*W_2\|_{\C^\infty}
\lessapprox_\kappa
\|W_1\|_{L_t^{\widetilde{q_1},\infty}L_x^2}\|W_2\|_{L_t^{\widetilde{q_1},\infty}L_x^2},
\end{equation}
where $s_1 = \frac{d}{2}- \frac{2}{q_1}$.
\end{proposition}
The proof of Proposition \ref{proposition:UpperHalfC2Cinfty} rests on the following key lemma. To emphasise that Theorem \ref{t:maximalspace} readily extends to fractional Schr\"odinger propagators, we state the result at such a level of generality. 
\begin{lemma}[Dispersive estimate]\label{l:GLNY}
Let $m\in (0,\infty)\backslash\{1\}$, $d\geq1$, $0<q<\infty$ and 
\[
s=\frac d2-\frac m q,
\qquad
\frac2q\leq \frac d2.
\] 
Then, we have 
\begin{equation}
\label{e:dispersive}
\bigg|
\int e^{i(x \cdot\xi+t|\xi|^m)}|\xi|^{-2s+ i \kappa} \chi(|\xi|)\,\d\xi
\bigg|
\lessapprox_\kappa
|t|^{-\frac2q}
\end{equation}
and
\begin{equation}\label{e:dispersiveloc}
\bigg|
\int e^{i(x \cdot\xi+t|\xi|^m)}\varphi(|\xi|)\,\d\xi
\bigg|
\lesssim
(1+|t|)^{-\frac d2}.
\end{equation}
\end{lemma}

The estimate \eqref{e:dispersive} with $m \geq 2$ can be found in Kenig--Ponce--Vega \cite[Lemma 3.4]{KPV}. These estimates have also featured more recently in \cite{GLNY} and \cite{BLN_Forum}, where a proof in the general case
 $m \in (0,\infty) \setminus \{0\}$ is given. The frequency-local estimate \eqref{e:dispersiveloc} is more standard (an explanation can be found, for example, in the forthcoming proof of Lemma \ref{l:disp space frec loc}).

The following Lorentz space refinement of the Hardy--Littlewood--Sobolev inequality by O'Neil \cite{ONeil} will also play an important role.
\begin{lemma}[\cite{ONeil}]\label{l:ONeil}
Let $d \geq 1$, $0<\lambda<d$, $1<p_1, p_2<\infty$, and $1\leq a \leq\infty$ satisfy
\[
\frac{1}{p_1} +\frac{1}{p_2}+\frac\lambda d=2.
\]
Then
\[
\left|\int_{\mathbb R^d}\int_{\mathbb R^d}f(x)g(y)|x-y|^{-\lambda}\,\d x\d y\right|
\lesssim 
\|f\|_{L^{p_1,a}}\|g\|_{L^{p_2,a'}}
\]
for all $f\in L^{p_1,a}(\mathbb R^d)$ and $g\in L^{p_2,a'}(\mathbb R^d)$.
\end{lemma}

\begin{proof}[Proof of Proposition \ref{proposition:UpperHalfC2Cinfty}]
We first show \eqref{i:C2}, which means we assume $d \geq 1$ and $4 < q_0 < \infty$. Note that the operator $W_1U_{\leq 1}|D|^{-2s_0 + i\kappa}U_{\leq 1}^*W_2$ is given by
\[
W_1U_{\leq 1}|D|^{-2s_0}U_{\leq 1}^*W_2\phi(t,x) 
= 
\int_{\mathbb R^{1+d}}\phi(t',x')W_1(t,x)W_2(t',x')K(t-t',x-x')\,\d t'\d x',
\]
where $K(t,x)=\int e^{i(x \cdot\xi+t|\xi|^2)}|\xi|^{-2s+ i \kappa}\,\d\xi$.
Thus, by Lemmas \ref{l:GLNY} and \ref{l:ONeil}, we have
\begin{align*}
\|W_1U_{\leq 1}|D|^{-2s_0 + i\kappa}U_{\leq 1}^*W_2\|_{\C^2}^2 
& = \int 
\bigl| W_1(t,x)W_2(t',x')K(t-t',x-x') 
\bigr|^2 \,\d t'\d x' \d t \d x \\
& \lessapprox_\kappa
\int_{\mathbb R}\int_{\mathbb R}
\|W_1(t,\cdot)\|_{L_x^2}^2 \|W_2(t',\cdot)\|_{L_x^2}^2|t-t'|^{-4/q_0}\,\d t\d t'\\
& \lesssim \|W_1\|_{L_t^{\widetilde{q_0},4}L_x^2}\|W_2\|_{L_t^{\widetilde{q_0},4}L_x^2}.
\end{align*}

Next we show \eqref{i:Cinfty}, so we assume either $d \geq 2$ and $2 < q_1 < \infty$, or $d = 1$ and $4 \leq q_1 < \infty$. First note that use of Lemmas \ref{l:GLNY} and \ref{l:ONeil} again yields
\begin{align*}
\|U_{\leq 1}|D|^{-2s_1+i\kappa}U_{\leq 1}^* \phi\|_{L_t^{q_1, 2} L_x^{\infty}} 
& = \bigg\| \int K(t-t',x-x') \phi(t',x') \d t' \d x' \bigg\|_{L_t^{q_1, 2} L_x^{\infty}}\\
& \lessapprox_\kappa  \bigg\| \int_{\mathbb R} |t-t'|^{-2/q_1} \| \phi(t', \cdot)\|_{L_x^1} \d t' \bigg\|_{L_t^{q_1, 2}}\\
& \lesssim \|\phi \|_{L_t^{q_1',2}L_x^1}.
\end{align*}
Thus, by the H\"{o}lder inequality twice,
\begin{align*}
\|W_1U_{\leq 1}|D|^{-2s_1+i\kappa}U_{\leq 1}^*W_2 \phi\|_{L^2}
& \lesssim\|W_1\|_{L_t^{\widetilde{q_1},\infty}L_x^2} \|U_{\leq 1}|D|^{-2s_1+i\kappa}U^* W_2 \phi\|_{L_t^{q_1, 2}} \\
& \lessapprox_\kappa
\|W_1\|_{L_t^{\widetilde{q_1},\infty}L_x^2}\|W_2\|_{L_t^{\widetilde{q_1},\infty}L_x^2} \|\phi\|_{L^2}
\end{align*}
and this gives \eqref{i:Cinfty}.
\end{proof}

\subsection{Proof of \eqref{i:main dual} for $2 < q \leq 4$}

Proposition \ref{proposition:UpperHalfC2Cinfty} and an analytic interpolation imply the following estimates, which are slightly better than the desired estimates \eqref{i:main dual} thanks to the embedding of Lorentz spaces.
\begin{proposition}\label{proposition:Schatten_UpperHalf}
Let $d \geq 2$, $2 < q\leq 4$ and $s= \frac{d-2}{2} + \frac{2}{\widetilde{q}}$. 
For $\beta'>\frac{\widetilde{q}}{2}$ we have
\begin{equation}\label{est:goal-prop3.2}
\| W_1 U_{\leq 1} |D|^{- 2s} U_{\leq 1}^* W_2 \|_{\mathcal{C}^{\beta'}} \lesssim \|W_1 \|_{L_t^{\widetilde{q}, 2\beta'} L_x^2} \|W_2 \|_{L_t^{\widetilde{q}, 2\beta'} L_x^2}.
\end{equation}
\end{proposition}

\begin{proof}
We fix $\phi,\psi \in \mathcal{B}$ and note that, according to Lemma \ref{l:Simon}, it suffices to show \begin{equation}\label{est:propSchatten_UpperHalf-1}
\|\langle W_1 U_{\leq 1} |D|^{- 2s} U_{\leq 1}^* W_2 \phi_j,  \psi_j \rangle_{L^2(\R^{d+1})} \|_{\ell^{\beta'}} \lesssim \|W_1 \|_{L_t^{\widetilde{q}, 2\beta'} L_x^2} \|W_2 \|_{L_t^{\widetilde{q}, 2\beta'} L_x^2}.
\end{equation}
If $q_0, q_1$ satisfy $2<\widetilde{q_0}<4$ and $2 < \widetilde{q_1} < \infty$, then
Proposition \ref{proposition:UpperHalfC2Cinfty} and Lemma \ref{l:Simon} imply that
\begin{equation}
\label{est:prop3.1-01}
\|\langle W_1 U_{\leq 1} |D|^{- 2s_0 + i \kappa}  U_{\leq 1}^* W_2 \phi_j,  \psi_j \rangle_{L^2(\R^{d+1})} \|_{\ell^{2}}
\lessapprox_\kappa \|W_1 \|_{L_t^{\widetilde{q_0}, 4} L_x^2} \|W_2 \|_{L_t^{\widetilde{q_0}, 4} L_x^2}
\end{equation}
and
\begin{equation}
\label{est:prop3.1-02}
\|\langle  W_1 U_{\leq 1} |D|^{- 2s_1 + i \kappa}  U_{\leq 1}^* W_2 \phi_j,  \psi_j \rangle_{L^2(\R^{d+1})} \|_{\ell^{\infty}}
\lessapprox_\kappa \|W_1 \|_{L_t^{\widetilde{q_1}, \infty} L_x^2} \|W_2 \|_{L_t^{\widetilde{q_1}, \infty} L_x^2}.
\end{equation}
Here, $s_i =\frac{d}{2}-\frac{2}{q_i}$ for $i = 0,1$. We would like to perform an analytic interpolation on these estimates with $\frac{1}{q_1} = \frac{1}{2} - \delta$ and $\frac{1}{q_0} = \frac{1}{4} - \delta$, where $\delta = \frac{1}{2}(\frac{1}{\beta} - \frac{2}{q}) > 0$.

Denote $\mathfrak{S}^{\circ} := \{z \in \Co \, : \, -d < \Re z < 0\}$  and consider the function on the open strip $\mathfrak S^\circ$ given by
\[
F(z)
=
\langle W_1 U_{\leq 1} |D|^z U_{\leq 1}^* W_2 \phi, \psi \rangle_{L_{t,x}^2},
\]
where 
$W_1, W_2 \in L^1(\R^{d+1})$ are simple and 
$\phi, \, \psi \in L^2(\R^{d+1})$ are normalized (i.e. $\|\phi\|_{L^2} = \|\psi\|_{L^2} =1$). Since $-2s_i \in (-d,0)$ for $i = 0,1$, it follows that if we can show that $F$ is analytic on $\mathfrak S^\circ$, then  
the family of bilinear operators $\{T_z\}_{z \in \mathfrak{S}' } : (L_t^{\widetilde{q_0}, 4} L_x^2 \cap L_t^{\widetilde{q_1}, \infty} L_x^2)^2 \to \ell^2$ defined by
\[
T_z(\phi,\psi) = \{\langle W_1 U_{\leq 1} |D|^{z} U_{\leq 1}^* W_2 \phi_j,  \psi_j \rangle_{L^2(\R^{d+1})} \}_j
\]
with $\mathfrak{S}' = \{z \in \mathbb{C} : -2s_0 \leq \Re z \leq -2s_1 \}$ is analytic. Consequently, a bilinear analytic interpolation argument (see, for example, \cite{CJ} or more recent work in \cite{GM14}, \cite{GO22}) with \eqref{est:prop3.1-01} and \eqref{est:prop3.1-02} implies \eqref{est:propSchatten_UpperHalf-1}.

To see that $F$ 
is analytic on $\mathfrak{S}^{\circ}$, we first use Parseval's identity to write
\begin{align*}
F(z)
= \langle  |D|^z U_{\leq 1}^* W_2 \phi, U_{\leq 1}^* \overline{W}_1 \psi \rangle_{L_{x}^2}  \simeq \bigl\langle |\xi|^z \F_{x}\bigl(U_{\leq 1}^* W_2 \phi\bigr), 
\F_{x}\bigl(U_{\leq 1}^* \overline{W}_1 \psi \bigr) \bigr\rangle_{L_{\xi}^2},
\end{align*}
where $\F_{x}$ denotes the (spatial) Fourier transform.
Since $\mathfrak{S}^{\circ}$ is open, for any $z_0 \in \mathfrak{S}^{\circ}$, there exists $0<c \ll 1$ such that
\begin{equation}\label{est:condition_c}
\Re z_0 + d > 3 c, \quad \Re z_0 < - 3c.
\end{equation}
Since $|\xi|^z$ is analytic as a map $\mathfrak{S}^{\circ} \to \Co$ for any $\xi \in \R^d \setminus \{0\}$, by the dominated convergence theorem, it suffices to see that there exists a non-negative function $g$ on $\R^d$ such that
\begin{equation*}\label{est:condition1_g}
\sup_{|h| < c} \bigg| \frac{|\xi|^{z_0+h} - |\xi|^{z_0}}{h} \bigg| \leq g(\xi) \qquad \text{for any $\xi \in \R^d\setminus \{0\}$},
\end{equation*}
and 
\begin{equation}\label{est:condition2_g}
\bigl\langle g \bigl|\F_{x}\bigl(U_{\leq 1}^* W_2 \phi\bigr)\bigr|, 
\bigl|\F_{x}\bigl(U_{\leq 1}^* \overline{W}_1 \psi \bigr) \bigr| \bigr\rangle_{L_{\xi}^2} < \infty.
\end{equation}
In fact, whenever $|h| < c$ we have
\begin{align*}
 \bigl| |\xi|^{z_0 + h} - |\xi|^{z_0} \bigr| 
& \leq |h| |\log |\xi|| \sup_{|z - z_0| < c} |\xi|^{\Re z} \lesssim_c |h| (|\xi|^{-c} + |\xi|^c) \sup_{|z - z_0| < c} |\xi|^{\Re z}
\end{align*}
so, thanks to \eqref{est:condition_c}, it suffices to show \eqref{est:condition2_g} with $g(\xi) = |\xi|^{-2s}$ and $2s = c$ or $2s = d-c$. By the Cauchy--Schwarz inequality and Plancherel's identity, it suffices to show
\[
\| U_{\leq 1}^* |D|^{-s}  W \phi\|_{L_{x}^2} < \infty
\]
for simple functions $W$ and $\phi \in L^2$. However, by \eqref{e:classicalS} we have
\[
\| U_{\leq 1}^* |D|^{-s}  W \phi\|_{L_{x}^2} \lesssim \|W \phi\|_{L_t^1 L_x^{r'}} \leq \|W\|_{L_t^2 L_x^{\frac{d}{s}}} \| \phi\|_{L^2_tL^2_x} < \infty
\]
where $r \in (2,\infty)$ is given by $s = d(\frac{1}{2} - \frac{1}{r})$ (note $\frac{c}{2}$ and $\frac{d-c}{2}$ both belong to $(0,\frac{d}{2})$).
\end{proof}

\section{Proof of Theorem \ref{t:maximalspace} for $4 < q < \infty$} \label{s:lower half}

To complete the proof of Theorem \ref{t:maximalspace}, we consider the remaining cases where $d \geq 1$ and $4<q<\infty$. By duality, our goal is to show \eqref{i:main dual} for $\beta'>\frac{\widetilde{q}}{2}$ and $2<\widetilde{q}<4$. As explained in Section \ref{s:Motivation Outline}, the basic strategy is to interpolate between \eqref{i:C middle} and \eqref{i:C bottom}. The former estimates were proved in Proposition \ref{proposition:UpperHalfC2Cinfty} and next we prove the latter.

\subsection{Proof of \eqref{i:C bottom}}

\begin{proposition}
\label{prop:Proposition_4.1}
Let $d\geq1$, $2<\widetilde{q}<4$ and $s=\frac{d-2}{2}+\frac{2}{\widetilde{q}}$. For $\beta'>\frac{\widetilde{q}}{2}$ we have 
\begin{equation*}
    \|W_1U_{\leq 1}|D|^{-2s+i\kappa}U_{\leq 1}^*W_2\|_{\C^{\beta'}}
\lessapprox_\kappa
    \|W_1\|_{L_t^{\widetilde{q},2}L_x^2}\|W_2\|_{L_t^{\widetilde{q},2}L_x^2}.
\end{equation*}
\end{proposition}
In order to prove Proposition \ref{prop:Proposition_4.1}, we show the following.
\begin{lemma} \label{l:freqlocal}
Let $s \in \mathbb{R}$, $k \in \mathbb{Z}$. Then 
\begin{align}\label{i:C1 loc}
    \|W_1U_{\leq 1}|D|^{-2s+i\kappa}P_kU_{\leq 1}^*W_2\|_{\C^1}
    & \lesssim
    2^{k(d-2s)}
    \|W_1\|_{L_t^{2}L_x^2}\|W_2\|_{L_t^{2}L_x^2}
\end{align}
and if $\widetilde{q_1},\widetilde{q_2}\in (2,\infty)$ are such that $\frac{1}{\widetilde{q_1}}+\frac{1}{\widetilde{q_2}}>\frac12$, then
\begin{align}\label{i:C2 loc}
    \|W_1U_{\leq 1}|D|^{-2s+i\kappa}P_kU_{\leq 1}^*W_2\|_{\C^2}
    & \lessapprox_\kappa
    2^{k(d-2+\frac{2}{\widetilde{q_1}}+\frac{2}{\widetilde{q_2}}-2s)}
    \|W_1\|_{L_t^{\widetilde{q_1}}L_x^2}\|W_2\|_{L_t^{\widetilde{q_2}}L_x^2}.
\end{align}
\end{lemma}
\begin{proof}[Proof of Lemma \ref{l:freqlocal}]
Starting with \eqref{i:C1 loc}, thanks to Lemma \ref{l:Simon}, it is sufficient to check
\begin{align*}
\|\langle W_1U_{\leq 1}|D|^{-2s+i\kappa}P_kU_{\leq 1}^*W_2 \phi_j, \psi_j \rangle_{L^2_{t,x}} \|_{\ell^1}
\lesssim
    2^{k(d-2s)}
    \|W_1\|_{L_t^{2}L_x^2}\|W_2\|_{L_t^{2}L_x^2}
\end{align*}
uniformly in $\phi,\psi \in \mathcal{B}$. In fact, we can reduce to the case $k=0$ since
\begin{align*}
& \|\langle W_1U_{\leq 1}|D|^{-2s+i\kappa}P_kU_{\leq 1}^*W_2 \phi_j, \psi_j \rangle_{L^2_{t,x}} \|_{\ell^1} \\
& = 2^{k(d-2s)} \|\langle W_1^{(k)}U_{\leq 1}|D|^{-2s+i\kappa}P_0U_{\leq 1}^*W_2^{(k)} \phi_j^{(k)}, \psi_j^{(k)} \rangle_{L^2_{t,x}} \|_{\ell^1},
\end{align*}
where $F^{(k)}$ denotes the $L^2$-normalized function given by $F^{(k)}(t,x) = 2^{-k\frac{d+2}{2}} F(\frac{t}{2^{2k}},\frac{x}{2^k})$. So now we fix $k = 0$ and $\phi,\psi \in \mathcal{B}$, and note 
\begin{align*}
\|\langle W_1U_{\leq 1}|D|^{-2s+i\kappa}P_0U_{\leq 1}^*W_2 \phi_j, \psi_j \rangle_{L^2_{t,x}} \|_{\ell^1}
=
\|\langle |D|^{-2s+i\kappa}P_0U_{\leq 1}^*W_2\phi_j,\widetilde{P}_0U_{\leq 1}^*\overline{W}_1\psi_j\rangle_{L_x^2}\|_{\ell^1},
\end{align*}
where $\widetilde{P}_0$ is given by $\mathcal{F}(\widetilde{P}_0f)(\xi) = \vartheta(\xi)\widehat{f}(\xi)$ and the bump function $\vartheta$ is chosen so that $P_0 = \widetilde{P}_0 P_0$. Furthermore
\begin{align*}
|D|^{-2s+i\kappa}P_0U_{\leq 1}^* W_2\phi_j(x) 
& \simeq \langle \phi_j, \overline{W}_2 \tau_x \Phi \rangle_{L^2_{t,x}},
\end{align*}
where $\tau_x \Phi(t',x') := \Phi(t',x-x')$ and 
\[
\Phi(t,x) := \int |\xi|^{-2s + i\kappa} \varphi(\xi)e^{it|\xi|^2} e^{-ix\cdot \xi}  \,\mathrm{d}\xi. 
\]
Similarly, 
\begin{align*}
\widetilde{P}_0U_{\leq 1}^* \overline{W}_1\psi_j(x) & \simeq \langle \psi_j, W_1 \tau_x \Theta \rangle_{L^2_{t,x}},
\end{align*}
where $\Theta(t,x) := \int  \vartheta(\xi)e^{it|\xi|^2} e^{-ix\cdot \xi}  \,\mathrm{d}\xi.
$
In particular, observe that $\|\Phi(t,\cdot)\|_{L^2} \simeq \|\Theta(t,\cdot)\|_{L^2} \simeq 1$. Thus, by the Cauchy--Schwarz inequality (twice) sandwiched by use of Bessel's inequality we obtain
\begin{align*}
\|\langle |D|^{-2s+i\kappa}P_0U_{\leq 1}^*W_2\phi_j,\widetilde{P}_0U_{\leq 1}^*\overline{W}_1\psi_j\rangle_{L_x^2}\|_{\ell^1} &\leq
\int_{\mathbb R^d}
\| W_2 \tau_x\Phi \|_{L^2_{t',x'}}
\| W_1 \tau_x\Theta \|_{L^2_{t',x'}}
\,\d x \\
& \leq 
\bigg(\int_{\mathbb R^d}
\| W_2 \tau_x\Phi \|_{L^2_{t',x'}}^2 \,\d x \bigg)^{\frac{1}{2}}
\bigg(\int_{\mathbb R^d}
\| W_2 \tau_x\Phi \|_{L^2_{t',x'}}^2 \,\d x \bigg)^{\frac{1}{2}} \\
& \simeq \|W_1\|_{L^2_{t',x'}}\|W_2\|_{L^2_{t',x'}}
\end{align*}
as desired.

For \eqref{i:C2 loc}, one may again reduce to the case $k = 0$ by a rescaling argument. When $k = 0$ we argue more or less as in the proof of \eqref{i:C2}. In particular, applying the dispersive estimate \eqref{e:dispersiveloc} (instead of \eqref{e:dispersive}) and the Young convolution inequality, we have
\begin{align*}
    \|W_1U_{\leq 1}|D|^{-2s+i\kappa}P_0U_{\leq 1}^*W_2\|_{\C^2}^2
    & \lessapprox_\kappa
\iint\|W_1(t,\cdot)\|_{L^2}^2\|W_2(t',\cdot)\|_{L^2}^2(1+|t-t'|)^{-d}\,\d t\d t'\\
    & \lesssim
    \|W_1\|_{L^{\widetilde{q_1}}_tL^2_x}^2
    \|W_2\|_{L^{\widetilde{q_2}}_tL^2_x}^2
\end{align*}
whenever $\widetilde{q_1},\widetilde{q_2}\in (2,\infty)$ and\footnote{A slightly larger range of exponents is allowable for $d \geq 2$ but this seems to be of no advantage to us.} $\frac{1}{\widetilde{q_1}}+\frac{1}{\widetilde{q_2}}>\frac12$. 
\end{proof}

\begin{proof}[Proof of Proposition \ref{prop:Proposition_4.1}]
Now fix $\widetilde{q_*}\in(2,4)$ and $\beta'_* \in (1,2)$ satisfying $\beta'_*>\frac{\widetilde{q_*}}{2}$, and set $s_*:=\frac{d-2}{2}+\frac{2}{\widetilde{q_*}}$. Our goal is 
\begin{equation} \label{e:KTgoal}
    \|W_1U_{\leq 1}|D|^{-2s_*+i\kappa}U_{\leq 1}^*W_2\|_{\C^{\beta_*'}}
\lessapprox_\kappa
    \|W_1\|_{L_t^{\widetilde{q}_*,2}L_x^2}\|W_2\|_{L_t^{\widetilde{q}_*,2}L_x^2}.
\end{equation}

By complex interpolation between \eqref{i:C1 loc} and \eqref{i:C2 loc}, it follows that 
\begin{equation}\label{i:Cbeta'}
    \|W_1U_{\leq 1}|D|^{-2s_*+i\kappa}P_kU_{\leq 1}^*W_2\|_{\C^{\beta_*'}}
\lessapprox_\kappa
    2^{-\nu k} 
    \|W_1\|_{L_t^{\widetilde{q_1}}L_x^2}\|W_2\|_{L_t^{\widetilde{q_2}}L_x^2}
\end{equation}
for $(\frac{1}{\widetilde{q_1}},\frac{1}{\widetilde{q_2}}) \in \Delta_*$, where
\begin{align*}
\nu & = \nu(\tfrac{1}{\widetilde{q}_1},\tfrac{1}{\widetilde{q}_2}) := 2\bigg(1 - \frac{1}{\widetilde{q}_1} - \frac{1}{\widetilde{q}_2}\bigg) - d + 2s_* \\
\Delta_* & := \{(c_1,c_2)\in (0,\tfrac{1}{2})^2: c_1 + c_2 >\tfrac{1}{\beta_*'}\}.
\end{align*}
Note that a sufficiently small neighbourhood of $(\frac{1}{\widetilde{q_*}},\frac{1}{\widetilde{q_*}})$ is contained in $\Delta_*$ (see Figure \ref{f:3d}). 
\begin{figure}[t]
\begin{center}
\begin{tikzpicture}[rotate around x=-100, rotate around y=0,rotate around z=-5,scale=5]
\fill [black!20!,opacity=0.5](0,0,0)--(1,0,0)--(1,1,0)--(0,1,0)--(0,0,0);
\draw (0,0,0)--(1,0,0)--(1,1,0)--(0,1,0)--(0,0,0);

\draw (0,0,1)--(1,0,1)--(1,1,1)--(0,1,)--(0,0,1);

\fill [cyan,opacity=0.2](1,0,0)--(1,1,0)--(1,1,1)--(0,1,0)--(1,0,0);

\draw (0,0,0)--(0,0,1);
\node [left] at (0,0,0) {$2$};
\node [left] at (0,0,0.7) {$\frac {\widetilde{q_*}}{2}$};
\node [left] at (0,0,1) {$1$};
\node [left] at (0,0,0.6) {$\beta'_*$};

\draw [dashed](0,1,0)--(1,1,1);
\draw [dotted, thick](0,0,0)--(1,1,0);

\node [below] at (0,0,0) {$O$};
\node [below] at (1,0,0) {$\frac12$};
\node [above left] at (0,1,0) {$\frac12$};

\fill (1,1,1) circle (0.4pt);
\draw [dashed](1,0,0)--(0,1,0);
\draw [dashed](1,1,0)--(1,1,1);

\path [name path=x] (1,0,0)--(1,1,1);
\path [name path=y] (0,1,0)--(1,1,1);
\path [name path=midx] (1,0,0.7)--(1,1,0.7);
\path [name path=midy] (0,1,0.7)--(1,1,0.7);
\path [name intersections={of=x and midx,by={X}}];
\path [name intersections={of=y and midy,by={Y}}];
\path [name path=XY] (X)--(Y);
\path [name path=semidiag] (1/2,1/2,0)--(1,1,1);
\path [name intersections={of=semidiag and XY,by={M}}];

\path [name path=midx2] (1,0,0.6)--(1,1,0.6);
\path [name path=midy2] (0,1,0.6)--(1,1,0.6);
\path [name intersections={of=x and midx2,by={X2}}];
\path [name intersections={of=y and midy2,by={Y2}}];


\draw [dashed](1/2,1/2,0)--(M);

\fill [magenta!40!,opacity=0.8](0,0,0.6)--(1,0,0.6)--(1,1,0.6)--(0,1,0.6)--(0,0,0.6);
\draw [](0,0,0.6)--(1,0,0.6)--(1,1,0.6)--(0,1,0.6)--(0,0,0.6);

\fill [violet!80!,opacity=0.5] (X2)--(Y2)--(1,1,0.6)--(X2);

\draw [->] (1.8,0,0)--(1.8+0.3,0,0);
\node [right] at (1.8+0.3,0,0) {$\frac{1}{\widetilde{q_1}}$};
\draw [->] (1.8,0,0)--(1.8,0.3,0);
\node [above] at (1.8,0.3,0) {$\frac{1}{\widetilde{q_2}}$};
\draw [->] (1.8,0,0)--(1.8,0,0.3);
\node [above] at (1.8,0,0.3) {$\beta'$};

\coordinate (O) at (0,0,0);

\path [name path=sdx] (1,1,1-0.1)--(1,0+0.1,0);
\path [name intersections={of=sdx and midx2, by={X3}}];
\path [name path=sdy] (1,1,1-0.1)--(0+0.1,1,0);
\path [name intersections={of=sdy and midy2, by={Y3}}];
\path [name path=md] (0,0,0.6)--(1,1,0.6);
\path [name path=XY3] (X3)--(Y3);
\path [name intersections={of=md and XY3, by={M3}}];
\path [name intersections={of=md and semidiag, by={M2}}];

\draw [dashed] (1,1,1)--(M2);

\coordinate (M4) at (0.85,0.85);
\fill (M4) circle (0.4pt);

\draw [] (M)--(M4);
\draw [dotted,thick] (0.85,0,0)--(M4);
\draw [dotted,thick] (0,0.85,0)--(M4);
\node [below] at (0.8,0,0) {$\frac{1}{\widetilde{q_*}}$};
\node [above left] at (0,0.8,0) {$\frac{1}{\widetilde{q_*}}$};

\draw [](0,0,0.7)--(1,0,0.7)--(1,1,0.7)--(0,1,0.7)--(0,0,0.7); 

\fill  (M) circle (0.4pt);

\fill [violet!80!,opacity=0.5] (X)--(Y)--(1,1,0.7)--(X);

\draw [dashed](1,0,0)--(1,1,1);

\fill [yellow](M3) circle (0.4pt);

\end{tikzpicture}
\caption{
For a fixed $\widetilde{q}_*$, the intersection between the boundary $\frac{1}{\widetilde{q_1}}+\frac{1}{\widetilde{q_2}}=\frac{2}{\widetilde{q}_*}$ and the diagonal line $\frac{1}{\widetilde{q_1}}=\frac{1}{\widetilde{q_2}}$ is contained in the interior $\Delta_{\beta_*'}$ lying on the layer underneath.
}\label{f:3d}
\end{center}
\end{figure}
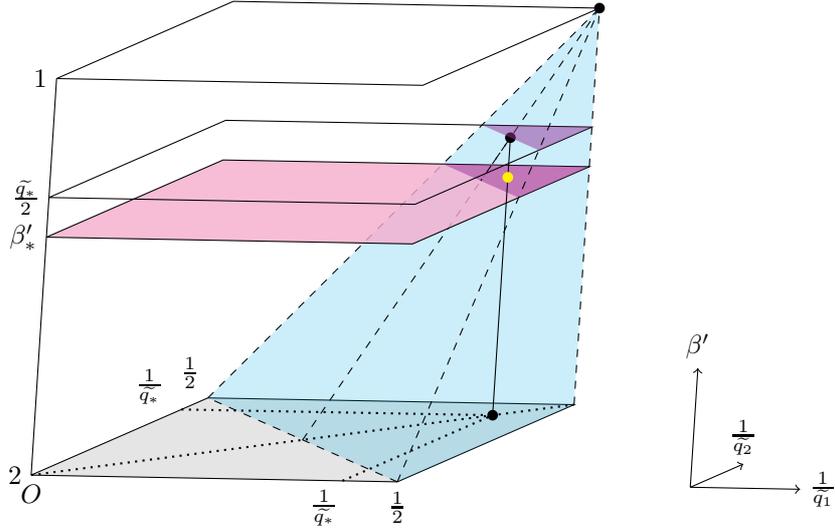
To sum up the estimates in \eqref{i:Cbeta'} to obtain \eqref{e:KTgoal}, we perform a bilinear interpolation argument in $\Delta_*$ (in the spirit of Keel--Tao \cite{KeelTao}). For this, we note that \eqref{i:Cbeta'} may be reinterpreted as 
\[
\|(W_1U_{\leq 1}|D|^{-2s_*+i\kappa}P_kU_{\leq 1}^*W_2)_k\|_{\ell_{\nu}^\infty(\C^{\beta_*'})}
\lessapprox_\kappa
\|W_1\|_{L_t^{\widetilde{q_1}}L_x^2}\|W_2\|_{L_t^{\widetilde{q_2}}L_x^2}, \qquad \nu = \nu(\tfrac{1}{\widetilde{q}_1},\tfrac{1}{\widetilde{q}_2})
\]
where $\ell^\infty_\nu(\C^{\beta_*'})$ is the weighted sequence space (of operators) with norm
\[
\|(T_k)_k\|_{\ell_{\nu}^\infty(\C^{\beta_*'})}
:= \sup_{k}2^{k\nu}\|T_k\|_{\C^{\beta_*'}}.
\]
Setting $T := (W_1U_{\leq 1}|D|^{-2s_*+i\kappa}P_kU_{\leq 1}^*W_2)_k$, we see that, in particular, $T$ is a bilinear operator which is bounded as follows:
\[
\begin{cases}
T:X_0\times X_0\to Y_0\\
T:X_0\times X_1\to Y_1\\
T:X_1\times X_0\to Y_1.
\end{cases}
\]
Here, $X_i := L^{a_i}_tL^2_x$ and $Y_i := \ell^\infty_{\nu_i}$ for $i=0,1$, and $\frac{1}{a_0}=\frac{1}{\widetilde{q}_*}+\delta$, $\frac{1}{a_1}=\frac{1}{\widetilde{q}_*}-2\delta$, $\nu_i := \nu(\tfrac{1}{a_0},\tfrac{1}{a_i})$. Also, $\delta > 0$ is chosen to be sufficiently small so that $(\tfrac{1}{a_j},\tfrac{1}{a_i}) \in \Delta_*$ for $(i,j) = (0,0), (1,0), (0,1)$ (see Figure \ref{f:nbh}). 
\begin{figure}[t]
\begin{center}
\begin{tikzpicture}[scale=8]
\fill [magenta, opacity=0.3] (0,0)--(1,0)--(1,1)--(0,1)--(0,0);
\draw (0,0)--(1,0)--(1,1)--(0,1)--(0,0);
\draw [dashed] (0,0)--(1,1);
\draw [dashed] (1,0)--(0,1);
\node [below left] at (0,0) {$O$};
\node [right] at (1,0) {$\frac{1}{\widetilde{q_1}}$};
\node [above] at (0,1) {$\frac{1}{\widetilde{q_2}}$};

\coordinate (Q) at (0.8,0.8);
\coordinate (A) at ([xshift=0.06cm,yshift=-0.12cm]Q);
\coordinate (B) at ([xshift=0.06cm,yshift=0.06cm]Q);
\coordinate (C) at ([xshift=-0.12cm,yshift=0.06cm]Q);


\coordinate (yC) at (O|-C);
\coordinate (yQ) at (O|-Q);
\coordinate (yA) at (O|-A);
\coordinate (xQ) at (Q|-O);

\fill [violet!70!, opacity=0.7] (1,1)--([xshift=-0.5cm,yshift=1cm]xQ)--([xshift=1cm,yshift=-0.5cm]yQ)--(1,1);

\fill [yellow](Q) circle (0.3pt);
\fill (A) circle (0.3pt);
\fill (B) circle (0.3pt);
\fill (C) circle (0.3pt);

\draw (A)--(B)--(C)--(A);
\draw [dotted] (A)--(A|-O);
\draw [dotted] ([xshift=1.1cm]O|-A)--(O|-A);
\draw [dotted] (C)--(C|-O);
\draw [dotted] ([xshift=1.1cm]O|-C)--(O|-C);
\draw [dotted] (Q)--(Q|-O);
\draw [dotted] (Q)--([xshift=1.1cm]O|-Q);

\node [below]at (Q|-O) {$\frac{1}{\widetilde{q_*}}$};
\node [below]at (A|-O) {$\frac{1}{a_0}$};
\node [left]at (O|-A) {$\frac{1}{a_1}$};
\node [below]at (C|-O) {$\frac{1}{a_1}$};
\node [left]at (O|-C) {$\frac{1}{a_0}$};

\node [below] at (1,0) {$\frac12$};
\node [left] at (0,1) {$\frac12$};

\draw[<->] ([xshift=1.1cm]yC)--([xshift=1.1cm]yQ);
\draw[<->] ([xshift=1.1cm]yA)--([xshift=1.1cm]yQ);
\node [above right] at ([xshift=1.1cm]yQ) {$\delta$};
\node [below right] at ([xshift=1.1cm]yQ) {$2\delta$};


\end{tikzpicture}
\caption{
The $(\frac{1}{\widetilde{q_1}},\frac{1}{\widetilde{q_2}})$-plane at height $\beta'=\beta_*'$.
}\label{f:nbh}
\end{center}
\end{figure}
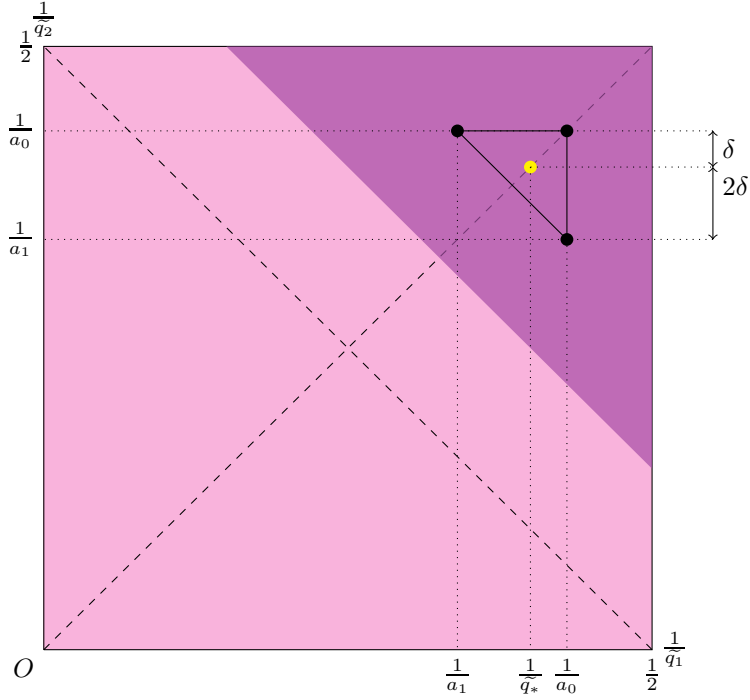
By a bilinear interpolation argument (see \cite[Exercise 5 (b)]{BL76}) it follows that $T$ is bounded as mapping
\[
T :(L_t^{a_0}L_x^2, L_t^{a_1}L_x^2)_{\frac{1}{3},2} \times(L_t^{a_0}L_x^2, L_t^{a_1}L_x^2)_{\frac{1}{3},2}
\to
(\ell_{\nu_0}^\infty(\C^{\beta_*'}),\ell_{\nu_1}^\infty(\C^{\beta_*'}))_{\frac{2}{3},1},
\]
which is equivalent to 
\[
T:L_t^{\widetilde{q_*},2}L_x^2\times L_t^{\widetilde{q_*},2}L_x^2\to\ell_0^1(\C^{\beta_*'}).
\]
Hence,
\[
\|W_1U_{\leq 1}|D|^{-2s_*+i\kappa}U_{\leq 1}^*W_2\|_{\C^{\beta_*'}}
\leq
\sum_{k}\|W_1U_{\leq 1}|D|^{-2s+i\kappa}P_kU_{\leq 1}^* W_2\|_{\C^{\beta_*'}}
\lessapprox_\kappa
\|W_1\|_{L_t^{\widetilde{q_*},2}L_x^2}\|W_2\|_{L_t^{\widetilde{q_*},2}L_x^2}
\]
which is \eqref{e:KTgoal}.

\end{proof}

\subsection{Proof of \eqref{i:main dual} for $4 < q < \infty$}

\begin{proposition}\label{proposition:Schatten_LowerHalf}
Let $d \geq 1$, $2< \widetilde{q} < 4$ and $s= \frac{d-2}{2} + \frac{2}{\widetilde{q}}$. For $\beta' > \frac{\widetilde{q}}{2}$, we have
\begin{equation}\label{est:goal-prop4.1}
\| W_1 U_{\leq 1} |D|^{- 2s} U_{\leq 1}^*W_2 \|_{\mathcal{C}^{\beta'}}
\lesssim 
\|W_1 \|_{L_t^{\widetilde{q}} L_x^2} \|W_2 \|_{L_t^{\widetilde{q}} L_x^2}.
\end{equation}
\end{proposition}

\begin{proof}
Suppose $\widetilde{q} \in (2,4)$, $s= \frac{d-2}{2} + \frac{2}{\widetilde{q}}$, and $\beta' > \frac{\widetilde{q}}{2}$. We set $\varepsilon := \frac{2}{\widetilde{q}} - \frac{1}{\beta'}$ and observe that it suffices to prove \eqref{est:goal-prop4.1} with $\beta'$ sufficiently close to $\frac{\widetilde{q}}{2}$ (i.e. $\varepsilon > 0$ sufficiently small). Also, we fix
$\phi, \psi \in \mathcal{B}$ and note that, from Lemma \ref{l:Simon}, it is enough to prove 
\begin{equation}\label{est:propSchatten_LowerHalf-1}
\|\langle W_1 U_{\leq 1} |D|^{- 2s} U_{\leq 1}^* \overline{W}_2 \phi_j,  \psi_j \rangle_{L^2(\R^{d+1})} \|_{\ell^{\beta'}} \lesssim \|W_1 \|_{L_t^{\widetilde{q}} L_x^2} \|W_2 \|_{L_t^{\widetilde{q}} L_x^2}.
\end{equation}
For this, we shall see that one can upgrade the estimates from Proposition \ref{prop:Proposition_4.1} by making use of the Lorentz space improvement in \eqref{est:prop3.1-01} and analytic interpolation.

Let $\widetilde{q}_0, \widetilde{q}_1$ be given by
\[
\frac{1}{\widetilde{q}_0} = \frac{1}{4} + \lambda_0 \varepsilon, \quad \frac{1}{\widetilde{q}_1} = \frac{1}{2} - \lambda_1\varepsilon
\]
where $\lambda_0 = (\frac{q}{4} - 1)\lambda_1$ and $\lambda_1 = (4 - \frac{16}{q})^{-1}$. Choosing $\varepsilon$ sufficiently small guarantees that $\widetilde{q}_0, \widetilde{q}_1 \in (2,4)$. Also, let $\beta_1 = \varepsilon^{-1}(1 - \frac{4}{q})$ and note that the choice of $\lambda_1$ ensures that $\beta_1' > \frac{\widetilde{q}_1}{2}$. Thanks to these choices, \eqref{est:prop3.1-01} implies
\begin{equation*}
\|\langle W_1 U_{\leq 1} |D|^{- 2s_0 + i \kappa}  U_{\leq 1}^* W_2 \phi_j,  \psi_j \rangle_{L^2(\R^{d+1})} \|_{\ell^{2}}
\lessapprox_\kappa \|W_1 \|_{L_t^{\widetilde{q_0}, 4} L_x^2} \|W_2 \|_{L_t^{\widetilde{q_0}, 4} L_x^2}
\end{equation*}
and Proposition~\ref{prop:Proposition_4.1} and Lemma \ref{l:Simon} imply
\begin{equation}\label{est:prop4.6-02}
\|\langle W_1 U_{\leq 1} |D|^{- 2s_1 + i \kappa} U_{\leq 1}^* W_2 \phi_j,  \psi_j \rangle_{L^2(\R^{d+1})} \|_{\ell^{\beta'_1}} \lessapprox_\kappa \|W_1 \|_{L_t^{\widetilde{q_1},2} L_x^2} \|W_2 \|_{L_t^{\widetilde{q_1},2} L_x^2}.
\end{equation}
Finally, we let $\theta \in (0,1)$ be given by $\theta = 1 - \frac{4}{q}$. By analytic interpolation (as in the proof of Proposition \ref{proposition:Schatten_UpperHalf}) and the choices of $q_0$ and $q_1$ above, one can check that \eqref{est:propSchatten_LowerHalf-1} now follows.
\end{proof}

\section{On the pointwise convergence problem}\label{s:PW}
In this section we consider (local) maximal-in-time estimates (Theorem \ref{t:maximal main}) and applications to the associated pointwise convergence problem (Corollary \ref{c:div set}). 
First of all, we give a precise definition of the density function of $\gamma \in \mathcal{C}^{\beta,s}$. 
Here, for $d\geq 1$, $s \geq 0$, 
\[
\mathcal{C}^{\beta,s} =
\{
\gamma \in \mathrm{Com}(\dot{H}^{-s} (\R^d), \dot{H}^{s} (\R^d)) : \| |D|^s \gamma |D|^s \|_{\C^{\beta}(L^2(\R^d))}<\infty
\},
\]
where $\mathrm{Com}(\dot{H}^{-s} (\R^d), \dot{H}^{s} (\R^d))$ denotes the set of compact operators from $\dot{H}^{-s}(\R^{d})$ to $\dot{H}^{s}(\R^{d})$. Using a finite-rank approximation, if $\gamma \in \mathcal{C}^{\beta,s}$ is self-adjoint then there exist orthonormal functions $(f_j )_j \subset L^2(\R^d)$ and $(\lambda_j)_j \in \ell^{\beta}$ such that
\[
\gamma^N := \sum_{j=1}^N \lambda_j |D|^{-s} \Pi_{f_j} |D|^{-s}, \quad \lim_{N \to \infty} \||D|^s(\gamma^N - \gamma)|D|^s \|_{\C^{\beta}}=0. 
\]
The density function of $\gamma^N$ is defined by
\[
\rho_{\gamma^N}(x) = \sum_{j=1}^N \lambda_j | |D|^{-s} f_j (x)|^2.
\]
We define the density function of $\gamma \in \mathcal{C}^{\beta,s}$ as a limit of $\rho_{\gamma^N}$ in the following manner. We claim that if $s > \frac{d}{2} - \frac{\alpha}{2 \beta}$, the density function $\rho_{\gamma}$ is well-defined in $L^1(\d \mu)$ for each $\alpha$-dimensional measure $\mu$. For this, it suffices to see
\begin{equation}
    \label{est:DensityWell-defined}
    \bigg\| \sum_{j} \lambda_j | |D|^{-s} f_j |^2 \bigg\|_{L^1(\d \mu)} \lesssim \|\lambda \|_{\ell^{\beta}}
\end{equation}
for all orthonormal functions $(f_j)_{j}$ in $L^2(\R^d)$ and $(\lambda_j)_j\in\ell^\beta$ since the estimate tells that $(\rho_{\gamma^N})_N$ is a Cauchy sequence in $L^1(\d \mu)$, and we may take $\rho_{\gamma} \in L^1(\d \mu)$ as the limit of this sequence. 

In order to verify \eqref{est:DensityWell-defined}, we employ the following two estimates:
\begin{equation}
\label{est:DensityWell-defined1}
    \bigg\| \sum_{j} \lambda_j | P_k f_j |^2 \bigg\|_{L^{\infty}(\d \mu)} \lesssim 2^{dk} \|\lambda \|_{\ell^{\infty}},
\end{equation}
and for $s>\frac{d}{2}- \frac{\alpha}{2}$, 
\begin{equation}
\label{est:DensityWell-defined2}
    \bigg\| \sum_{j} \lambda_j | P_k f_j |^2 \bigg\|_{L^{1}(\d \mu)} \lesssim 2^{2 s k} \|\lambda \|_{\ell^{1}}.
\end{equation}
For the proof of \eqref{est:DensityWell-defined1}, we refer the reader forward to the proof of \eqref{e:C1k}. 
For the second estimate, we note that Barcel\'o \emph{et al}. \cite[Appendix A]{BBCR} obtained 
\[
\bigl\| \sup_k |P_k g| \bigr\|_{L^2(\d \mu)} \lesssim \| g \|_{H^s(\R^d)}
\]
if $s >\frac{d}{2} - \frac{\alpha}{2}$, and this clearly implies \eqref{est:DensityWell-defined2} in the case $k\geq 0$. 
For $k \leq 0$, the inequality 
\begin{align*}
\sup_{x}|P_k g(x)| &
= \sup_{x} |((\F_{\xi}^{-1} \varphi_k) * g)(x)| \\
& \lesssim \|\varphi_k\|_{L^2} \|g\|_{L^2} \sim 2^{\frac{d}{2}k} \|g\|_{L^2}
\end{align*}
implies $\|P_k g\|_{L^2(\mathrm{d}\mu)} \lesssim 2^{\frac{d}{2}k} \|g\|_{L^2}$. Here, we use the notation $A \sim B$ when both $A \lesssim B$ and $B \lesssim A$ hold. From the above, we see that \eqref{est:DensityWell-defined2} holds in the case $k< 0$. 
Finally, by using Lemma~\ref{l:Bourgain's trick} together with \eqref{est:DensityWell-defined1} and \eqref{est:DensityWell-defined2}, we conclude \eqref{est:DensityWell-defined} for $s > \frac{d}{2} - \frac{\alpha}{2 \beta}$.

\subsection{Proof of Corollary \ref{c:div set}}
Corollary \ref{c:div set} may be deduced from the maximal-in-time estimates in Theorem \ref{t:maximal main} using well-established arguments (for example, \cite{BBCR, CL}). For the sake of completeness, we include a sketch for the case $m \in (1,\infty)$ (the case $m \in (0,1)$ can be handled in a similar manner).

First we note that an argument based on Frostman's lemma from geometric measure theory (see for example \cite{BBCR, CL}) means that the divergence set bound $\dim_H\mathfrak D(\gamma_0) \leq (d-2s)\beta$ follows if we can show that
\begin{equation} \label{e:mupointwise}
\lim_{t\to0} \rho_{\gamma(t)}(x)
=
\rho_{\gamma_0}(x)\quad \text{$\mu$-a.e. $x \in \mathbb{B}^d$}
\end{equation}
holds whenever $\mu \in \mathcal{M}^\alpha(\mathbb{B}^d)$ and $\alpha > (d-2s)\beta$.

Take $\gamma_0 \in \mathcal{C}^{\beta,s}$ to be self-adjoint and $\mu \in \mathcal{M}^\alpha(\mathbb{B}^d)$, where $s \in [\frac{d}{4},\frac{d}{2})$ and $\beta \in [1,\frac{\alpha}{d - 2s})$. Let $\gamma_0^N$ be defined as above, so that 
\begin{equation}
    \label{est:DensityConvergence}
  \lim_{N \to \infty}  \| \rho_{\gamma_0} - \rho_{\gamma_0^N} \|_{L^1(\d \mu)} =0.
\end{equation}
For each $t \in \mathbb{R}$, if we set $\gamma(t) = e^{it (-\Delta)^{m/2}} \gamma_0 e^{- it (-\Delta)^{m/2}}$ and $\gamma^N (t) = e^{it (-\Delta)^{m/2}} \gamma_0^N e^{- it (-\Delta)^{m/2}}$, then the fact that $e^{it (-\Delta)^{m/2}}$ is unitary means that $\rho_{\gamma(t)}$ is well-defined in $L^1(\d \mu)$ in the same manner.

In order to prove \eqref{e:mupointwise}, we first note that this holds in the finite-rank case; that is, for each fixed $N \in \mathbb{N}$, we have $\lim_{t \to 0} \rho_{\gamma^N(t)}(x) = \rho_{\gamma_0^N}(x)$ ($\mu$-a.e. $x \in \mathbb{B}^d$). Indeed, since we may write
\[
\rho_{\gamma^N(t)}(x) = \sum_{j=1}^N \lambda_j |e^{-it(-\Delta)^{m/2}}g_j (x)|^2
\]
for a certain orthonormal family $(g_j)_j$ in $\dot{H}^s(\mathbb{R}^d)$, the claim holds if $\lim_{t \to 0} e^{it(-\Delta)^{m/2}} f(x) = f(x)$ ($\mu$-a.e. $x \in \mathbb{B}^d$) whenever $f \in \dot{H}^s(\mathbb{R}^d)$. By standard arguments, this follows from the maximal estimate\footnote{Although this estimate may be known in certain cases, we note that it follows from \eqref{maximal}. However, we do not need the full power of \eqref{maximal}, and the special case $\beta = 1$ suffices.}
\begin{equation} \label{e:singlefunctionmaximal}
\|U_m f\|_{L_x^2(\mathbb B^d,\d\mu)L_t^\infty(0,1)} \lesssim \|f\|_{\dot{H}^s(\mathbb{R}^d)}.
\end{equation}

In order to extend to the infinite-rank case, we use Theorem \ref{t:maximal main}. It suffices to prove
\begin{equation} \label{e:oneoverk}
\mu(\{x\in \mathbb{B}^d:\limsup_{t\to 0}|\rho_{\gamma(t)}(x)-\rho_{\gamma_0}(x)|> k^{-1}\}) = 0
\end{equation}
for each integer $k \geq 1$, and to see this we fix $\varepsilon > 0$ and note
\begin{align*}
    &
    \mu(\{x\in \mathbb{B}^d:\limsup_{t\to 0}|\rho_{\gamma(t)}(x)-\rho_{\gamma_0}(x)|>k^{-1}\})\\
    &\leq
    \mu(\{x\in\mathbb{B}^d:\sup_{t \in (-1,1)}|\rho_{\gamma(t)}(x)-\rho_{\gamma^{N}(t)}(x)|>(3k)^{-1}\})\\
    &\quad+
    \mu(\{x\in\mathbb{B}^d:\limsup_{t\to 0}|\rho_{\gamma^{N}(t)}(x)-\rho_{\gamma_0^{N}}(x)|>(3k)^{-1}\})\\
    &\qquad+
    \mu(\{x\in\mathbb{B}^d: |\rho_{\gamma_0^{N}}(x)-\rho_{\gamma_0}(x)|>(3k)^{-1}\}) =: M_1 + M_2 + M_3,
\end{align*}
where $N$ is to be chosen momentarily. For $M_3$, by Chebyshev's inequality and \eqref{est:DensityConvergence}, we have
\[
M_3 \leq 3k \| \rho_{\gamma_0^{N}}-\rho_{\gamma_0}\|_{L_x^{1}(\d\mu)} < \varepsilon
\]
if we take $N = N(\varepsilon,k)$ sufficiently large. For $M_1$, we use Chebyshev's inequality and Theorem \ref{t:maximal main} to estimate
\[
M_1 \leq 3k \|\rho_{\gamma(t)}-\rho_{\gamma^{N_\varepsilon}(t)}\|_{L_x^{1}(\d\mu)L_t^\infty} \lesssim 3k \bigg(\sum_{j > N} |\lambda_j|^\beta \bigg)^{\frac{1}{\beta}} < \varepsilon,
\]
for $N= N(\varepsilon,k)$ sufficiently large. For $M_2$, by the above observation in the finite-rank case, it follows that $M_2 = 0$ for any choice of $N$. Hence we obtain \eqref{e:oneoverk}.

\subsection{The maximal-in-time estimates}
Here prove the estimate \eqref{maximal} in Theorem  \ref{t:maximal main} (the claims regarding sharpness in Theorem \ref{t:maximal main} are justified later in Section \ref{subsection:necssary}). We recall the notation $U_m=e^{it(-\Delta)^{m/2}}$. Also, throughout the following proof, we write $L_x^q(\d\mu)L_t^r = L_x^q(\mathbb R^d,\d\mu)L_t^r(0,1)$ and $\chi_E$ for the characteristic function of $E$.

We consider the following cases and treat them slightly differently.
\begin{itemize}
\item $m\in(1,\infty)$ and $s\in (\frac d4,\frac d2)$;
\item $m\in (0,1)$;
\item $m\in (1,\infty)$ and $s=\frac d4$.
\end{itemize}
The key oscillatory integral estimates to handle the first two cases are as follows.
\begin{lemma}\label{l:disp space frec loc}
Let $d\in\mathbb{N}$ and $\varphi\in C_0^\infty$ be supported in $\{r \in\mathbb R: 2^{-1} < r < 2\}$. For $m\in(0,\infty)\backslash\{1\}$, we have 
\begin{equation}\label{e:Sjolin}
\sup_{t\in \mathbb R}
\left|
\int_{\mathbb R^d} e^{i(x\cdot \xi+t|\xi|^m)}\varphi(2^{-k}|\xi|)\,\d\xi
\right|
\lesssim
\frac{2^{dk}}{(1+2^k|x|)^{\frac d2}}
\end{equation}
for each $k\in\mathbb Z$. For $m\in(0,1)$ and $|x|<1$, we further have that 
\begin{equation}\label{e:Walther}
\sup_{t\in(-1,1)}
\left|\int_{\mathbb R^d} e^{i(x\cdot\xi+t|\xi|^m)}\varphi(2^{-k}|\xi|)\,\d\xi\right|
\lesssim
\frac{2^{\frac d2 k}|x|^{-\frac d2}}{(1+2^k|x|^{\frac{1}{1-m}})^\frac d2}
\end{equation}
for each $k\in \mathbb Z$.
\end{lemma}
\begin{proof}
We change the variables to write 
\[
\int_{\mathbb R^d} e^{i(x\cdot \xi+t|\xi|^m)}\varphi(2^{-k}|\xi|)\,\d\xi
=
2^{dk}
\int_{\mathbb R^d} e^{i\theta_k(\xi)}\varphi(|\xi|)\,\d\xi
=:
\I_k,
\]
where $\theta_k(\xi)=2^k x\cdot\xi+2^{mk}t|\xi|^m$. Thus, \eqref{e:Sjolin} follows if
\[
\sup_{t \in \mathbb{R}} \bigg|\int_{\mathbb R^d} e^{i(x\cdot\xi+t|\xi|^m)}\varphi(|\xi|)\,\d\xi\bigg| \lesssim \frac{1}{(1 + |x|)^{\frac{d}{2}}}
\]
for any $x \in \mathbb{R}^d$. In fact, since $\{ (\xi,|\xi|^m) : |\xi| \in [\frac{1}{2},2]\}$ has non-vanishing gaussian curvature, we have (see, for example, \cite[Section VIII]{Stein})
\[
\bigg|\int_{\mathbb R^d} e^{i(x\cdot\xi+t|\xi|^m)}\varphi(|\xi|)\,\d\xi\bigg| \lesssim \frac{1}{(1 + |(x,t)|)^{\frac{d}{2}}}
\]
and the desired estimate follows immediately.

For \eqref{e:Walther} we fix $m \in (0,1)$ and $|x|, |t| < 1$, and choose constants $a$ and $b$ satisfying
\[
0 < a < \frac{1}{m} 2^{-(1-m)}, \qquad \frac{1}{m}2^{1-m} < b. 
\]
First we consider the case $2^k|x|^{\frac{1}{1-m}} \geq a^{-\frac{1}{1-m}}$. Then we have $2^{mk} \leq a 2^k|x|$, and therefore
\[
|\nabla \theta_k(\xi)| \geq 2^k|x| - m2^{mk}|t||\xi|^{m-1} \geq (1 - a m2^{1-m})2^k|x| \simeq 2^k|x|
\]
for $|\xi|$ in the support of $\varphi$. This means
\[
|\mathcal{I}_k| \leq C_N \frac{2^{dk}}{(2^k|x|)^N} \lesssim \frac{1}{|x|^{\frac{d(2-m)}{2(1-m)}}}
\]
as desired. The first inequality holds for any $N \in \mathbb{N}$ by integration by parts, and the second estimate can be checked since $2^{-k} \lesssim |x|^{\frac{1}{1-m}}$ and by taking $N$ sufficiently large (depending on $d$ and $m$).

In the remaining case $2^k|x|^{\frac{1}{1-m}} < a^{-\frac{1}{1-m}}$, the goal is
\begin{equation} \label{e:lastcase}
|\mathcal{I}_k| \lesssim \frac{2^{dk}}{(2^k|x|)^\frac{d}{2}}.
\end{equation}
We split into subcases:
\begin{itemize}
\item[(I)] $2^{mk}|t| < a 2^k|x|$ or $2^{mk}|t| > b 2^k|x|$;
\item[(II)] $a 2^k|x| \leq 2^{mk}|t| \leq b 2^k|x|$.
\end{itemize}
In Case (I) we have $|\nabla \theta_k(\xi)| \geq C2^k|x|$, where $C = 1 - a m2^{1-m}$, or $C = b m2^{-(1-m)} -1$, and therefore
\[
|\mathcal{I}_k| \leq C_N \frac{2^{dk}}{(1 + 2^k|x|)^N}
\]
for any $N \in \mathbb{N}$. This follows either from the trivial estimate $|\I_k|\lesssim 2^{dk}$ or by integration by parts. Taking $N > \frac{d}{2}$ we obtain \eqref{e:lastcase}.

In Case (II), we have
\[
|\det \text{Hess}\, \theta_k(\xi)|
\sim 
(2^{mk}|t|)^d
\sim
(2^k|x|)^d
\]
and therefore we again obtain \eqref{e:lastcase} by using standard results from the theory of oscillatory integrals (see, for example, \cite[Section VIII]{Stein}).
\end{proof}

The following elementary lemma will also be useful to us.
\begin{lemma}[\cite{ChoShiraki}]
\label{l:Young}
Let $0 < \alpha \leq d$ and $\mu\in \mathcal M^\alpha(\mathbb B^d)$. Then, for each $\ell\in\mathbb Z$ we have
\[
 \iint |g(x)| |h(x')| \chi_{(0,2^{-l})}(|x-x'|)\,  \d \mu(x) \d \mu(x') 
\lesssim 
2^{-\alpha l}\|g\|_{L_x^2(\mathbb B^d,\d\mu) }\|h\|_{L_x^2(\mathbb B^d, \d\mu)}.
\]
\end{lemma}
\begin{proof}
This is a simple consequence of the Cauchy--Schwarz inequality. Indeed, 
\begin{align*}
    &\iint |g(x)| |h(x')|\chi_{(2^{-l-1},2^{-l})}(x-x')\,\d\mu(x)\d\mu(x')\\
    &\leq
    \left(
    \int|g(x)|^2\mu(B(x,2^{-l}))\,\d\mu(x)
    \right)^\frac12
    \left(
    \int|h(x)|^2\mu(B(x,2^{-l}))\,\d\mu(x)
    \right)^\frac12\\
    &\lesssim
    2^{-\alpha l}\|g\|_{L^2(\mathbb B^d, \d\mu)}^2\|h\|_{L^2(\mathbb B^d, \d\mu)}^2.
\end{align*}
The defining property of the $\alpha$-dimensional measures was used at the last step. 
\end{proof}

\subsection*{The case of $m\in(1,\infty)$ and $\frac d4<s<\frac d2$}
\begin{equation}\label{e:goal maximal}
\|WT\overline{W}\|_{\C^{\beta'}}
\lesssim
\|W\|_{L^\infty(\d\mu)L^2}^2
\end{equation}
where $T=D^{-2s}U_mU_m^*$. It suffices to prove 
\begin{equation} \label{e:maximaltimegoal}
\sum_{k\in\mathbb Z} \|WT_k\overline{W} \|_{\C^{\beta'}}
\lesssim
\|W\|_{L^\infty(\d\mu)L^2}^2
\end{equation}
for $\beta\in[1,\frac{\alpha}{d-2s})$. Here, $T_k:=P_k^2T$ and $\widehat{P_k f}(\xi) = \varphi(2^{-k}|\xi|)\widehat{f}(\xi)$ with $\sum_{k\in\mathbb Z}\varphi(2^{-k}|\xi|)^2=1$. In order to prove \eqref{e:maximaltimegoal}, we make a further decomposition in the spatial variable. In particular, we write
\begin{align*}
T_k F(t,x)
&=
\int F(t',x') K_k(t-t',x-x')\,\d t' \d\mu(x')\\
&=
\sum_{l\geq0}\int F(t',x') K_{k,l}(t-t',x-x')\,\d t' \d\mu(x')
=:
\sum_{l\geq0}T_{k,l}F(t,x),
\end{align*}
where $K_{k,l}(t,x) = \chi_l(|x|)K_k(t,x)$, $\chi_l=\chi_{(2^{-l-1},2^{-l})}$, and 
\[
K_k(t,x)
=
\int e^{i(x\cdot\xi+t|\xi|^m)} \frac{\varphi(2^{-k}|\xi|)^2}{|\xi|^{2s}}\,\d\xi. 
\]
The key estimates are:
\begin{align}
\|WT_k\overline{W}\|_{\C^1} & \lesssim 2^{(d-2s)k}\|W\|_{L^2(\d\mu)L^2}^2 \label{e:C1k}; \\
\|WT_{k,l}\overline{W} \|_{\C^2} &\lesssim
\frac{2^{(d-2s)k}2^{-\frac{\alpha}{2}l}}{(1+2^{k-l })^\frac d2}\|W\|_{L^4(\d\mu)L^2}^2 \label{e:C2kl}; \\
\|WT_{k,l}\overline{W}\|_{\C^\infty}
& \lesssim
\frac{2^{(d-2s)k}2^{-\alpha l}}{(1+2^{k-l})^{\frac d2}}\|W\|_{L^\infty(\d\mu)L^2}^2. \label{e:Cinftykl}
\end{align} 
Before proving these, let us first see how one obtains \eqref{e:maximaltimegoal}.

As one application of \eqref{e:C1k} we see that
\begin{equation*} 
\sum_{k < 0} \|WT_k\overline{W}\|_{\C^{\beta'}} \lesssim \sum_{k < 0} \|WT_k\overline{W}\|_{\C^1} \lesssim \|W\|_{L^2(\d\mu) L^2}^2 \lesssim \|W\|_{L^\infty(\d\mu) L^2}^2.
\end{equation*}
To handle $k\geq0$, we shall first consider the case when $\frac{\alpha}{d-2s}>2$, in which case it suffices to consider $\beta \in (2,\frac{\alpha}{d-2s})$. Note that \eqref{e:C2kl} implies
\begin{align*}
\|WT_k\overline{W}\|_{\C^2}
&\lesssim
\sum_{l\leq k}\|WT_{k,l}\overline{W} \|_{\C^2}
+
\sum_{l>k}\|WT_{k,l}\overline{W}\|_{\C^2}\\
&\lesssim
2^{(d-2s)k}\|W\|_{L^4(\d\mu)L^2}^2
\left(
\sum_{l\leq k}\frac{2^{-\frac\alpha2 l}}{2^{\frac d2(k-l)}}
+
\sum_{l\leq k}2^{-\frac\alpha2 l}
\right)\\
&\lesssim
k2^{(\frac d2-2s-\frac{\alpha}{2})k}\|W\|_{L^4(\d\mu)L^2}^2.
\end{align*}
Interpolating this with \eqref{e:C1k} by H\"older's inequality, one obtains 
\begin{align*}
    \sum_{k\geq0} \|WT_k\overline{W}\|_{\C^{\beta'}}
    & \lesssim
    \sum_{k\geq0} \|WT_k\overline{W}\|_{\C^1}^{\frac{2}{\beta'}-1}     \|WT_k\overline{W}\|_{\C^2}^{2-\frac{2}{\beta'}} \\
    & \lesssim \sum_{k\geq0} k^{\frac{2}{\beta}} 2^{(d-2s - \frac{\alpha}{\beta})k} \|W\|_{L^4(\d\mu)L^2}^2
\end{align*}
which gives \eqref{e:maximaltimegoal} since $\beta < \frac{\alpha}{d - 2s}$. 

When $\frac{\alpha}{d-2s}\leq 2$, we interpolate between \eqref{e:C2kl} with \eqref{e:Cinftykl} before summing up in $l$. In particular we obtain 
\[
\|WT_{k,l}\overline{W} \|_{\C^{\beta'}}
\lesssim
\frac{2^{(d-2s)k}2^{-\frac\alpha\beta l}}{(1+2^{k-l})^\frac d2} \|W\|_{L^\infty(\d\mu)L^2}^2.
\]
Summing up both in $k,l\geq0$, and using that $s \in (\frac{d}{4},\frac{d}{2})$ and $\beta < \frac{\alpha}{d-2s}$, we obtain \eqref{e:maximaltimegoal} in this case too.

It remains to verify \eqref{e:C1k}--\eqref{e:Cinftykl}. For \eqref{e:C1k}, since $T_k = (D^{-s}P_kU_m)(D^{-s}P_kU_m)^*$, this is equivalent to
\[
\bigg\|\sum_j\lambda_j|D^{-s}P_kU_mf_j|^2\bigg\|_{L^\infty(\d\mu)L^\infty}
\lesssim
2^{k(d-2s)}\|\lambda\|_{\ell^\infty}
\]
for orthonormal systems $(f_j)_j$ in $L^2(\mathbb{R}^d)$. This latter estimate holds thanks to Bessel's inequality and since $\int |\xi|^{-2s} \varphi(2^{-k}\xi) \, \mathrm{d}\xi \sim 2^{k(d-2s)}$.

For \eqref{e:C2kl} and \eqref{e:Cinftykl}, we use \eqref{e:Sjolin} of Lemma \ref{l:disp space frec loc} to obtain
\[
|K_{k,l}(t,x)|
\lesssim
\chi_l(|x|)\frac{2^{(d-2s)k}}{(1+2^{k-l})^{\frac d2}}.
\]
Lemma \ref{l:Young} now immediately yields \eqref{e:C2kl}. For \eqref{e:Cinftykl}, note that the above kernel estimates gives
\begin{align*}
\|WT_{k,l}\overline{W}\|_{\C^\infty}
&=
\sup_{\|f_1\|_2=\|f_2\|_2=1}
\left|
\int f_1W(t,x)T_{k,l}\overline{W}f_2 (t,x)\,\d t\d \mu(x)
\right|\\
&\lesssim
\frac{2^{(d-2s)k}}{(1+2^{k-l})^\frac d2}
\sup_{\|f_1\|_2=\|f_2\|_2=1}\iint \|f_1W(\cdot,x)\|_{L^1}\|f_2\overline{W}(\cdot,x')\|_{L^1}\chi_l(|x-x'|)\,\d\mu(x)\d\mu(x'),
\end{align*}
and then \eqref{e:Cinftykl} follows from another application of Lemma \ref{l:Young}.

\subsection*{The case of $m\in(0,1)$}
In the case $2\alpha \leq d\beta$, the goal is to prove \eqref{e:maximaltimegoal} for $\frac{1}{2}(d-\frac{\alpha}{\beta}) < s < \frac{d}{2}$. Since $\frac{1}{2}(d-\frac{\alpha}{\beta}) \geq \frac{d}{4}$ in this case, and since \eqref{e:Sjolin} holds for $m \in (0,1)$ too, the above argument for $m \in (1,\infty)$ may be used to obtain \eqref{e:maximaltimegoal}.

Now suppose $2\alpha > d\beta$, in which case we want to prove \eqref{e:goal maximal} for $\frac d4-\frac{1}{2}(1-m)(\frac\alpha\beta-\frac d2) < s < \frac{d}{2}$. Our goal is again to show \eqref{e:maximaltimegoal} and the argument is similar to the above, except use of \eqref{e:Sjolin} is replaced by \eqref{e:Walther}. By doing so, one may obtain
\[
\|WT_{k,l}\overline{W} \|_{\C^2}
\lesssim
\frac{2^{(\frac d2-2s)k}2^{(\frac d2-\frac \alpha2)l}}{(1+2^{k-\frac{l}{1-m}})^\frac d2} \|W\|_{L^\infty(\d\mu)L^2}^2
\]
and 
\[
\|WT_{k,l}\overline{W} \|_{\C^\infty}
\lesssim
\frac{2^{(\frac d2-2s)k}2^{(\frac d2-\alpha)l}}{(1+2^{k-\frac{l}{1-m}})^\frac d2} \|W\|_{L^\infty(\d\mu)L^2}^2,
\]
which yield
\[
\|WT_{k,l}\overline{W}\|_{\C^{\beta'}}
\lesssim
\frac{2^{(\frac d2-2s)k}2^{(\frac d2-\frac\alpha\beta)l}}{(1+2^{k-\frac{l}{1-m}})^\frac d2}
\|W\|_{L^\infty(\d\mu)L^2}^2.
\]
Note that the restriction $2\alpha > d\beta$ implies, in particular, that $\beta < 2$. Therefore, 
\begin{align*}
\sum_{l\geq0} \|WT_{k,l}\overline{W} \|_{\C^{\beta'}}
&=
\sum_{l\leq(1-m)k} \|WT_{k,l}\overline{W} \|_{\C^{\beta'}}
+
\sum_{l>(1-m)k} \|WT_{k,l}\overline{W} \|_{\C^{\beta'}}\\
&\lesssim
2^{(\frac d2-2s)k}\|W\|_{L^\infty(\d\mu)L^2}^2
\left(
\sum_{l\leq (1-m)k}\frac{2^{(\frac d2-\frac \alpha\beta)l}}{2^{(k-\frac{l}{1-m})\frac{d}{2}}}
+
\sum_{l> (1-m)k} 2^{(\frac d2-\frac\alpha\beta)l}
\right)\\
&\lesssim
k2^{(\frac d2-2s+(1-m)(\frac d2-\frac\alpha\beta))k}\|W\|_{L^\infty(\d\mu)L^2}^2.
\end{align*}
This gives \eqref{subsection:maximalintime} since $\frac d2-(1-m)(\frac\alpha\beta-\frac d2)<2s$.

\subsection*{The case of $m\in(1,\infty)$ and $s=\frac d4$}
Instead of Lemma \ref{l:disp space frec loc}, we shall make use of the following estimate, which for $d=1$ appears in \cite{Sjolin} and is applied to the single-particle case by Barcel\'o \textit{et al.} \cite{BBCR}. 
To state it, we write $\varphi_{\leq N}(|\xi|)=\sum_{k\leq N} \varphi(2^{-k}|\xi|)$ for each $N\in \mathbb Z$. 

\begin{lemma}\label{l:disp space}
Let $d\in\mathbb N$. For $m\in(1,\infty)$, we have 
\[
\sup_{t\in\mathbb R}\left|\int_{\mathbb R^d} e^{i(x\cdot \xi+t|\xi|^m)}\frac{\varphi_{\leq N}(|\xi|)}{|\xi|^{\frac{d}{2}}}\,\d\xi\right|
\lesssim
|x|^{-\frac{d}{2}}
\]
uniformly in $N\in\mathbb Z$.
\end{lemma}

\begin{proof}
First notice that
\[
    \sum_{2^k<|x|^{-1}} \left|\int_{\mathbb R^d} e^{i(x\cdot\xi+t|\xi|^m)}      \frac{\varphi(2^{-k}|\xi|)}{|\xi|^{\frac{d}{2}}}
\, \mathrm{d} \xi\right|
    \lesssim
    \int_0^{|x|^{-1}} r^{\frac{d}{2}-1}\, \mathrm{d} r
    \sim 
    |x|^{-\frac{d}{2}}. 
\]
We set $V = \{k\in\mathbb Z:|x|^{-1} \leq 2^k \leq 2^N\}$ and for $k \in V$ we change variables to write
\[
\int_{\mathbb R^d} e^{i(x\cdot\xi+t|\xi|^m)}      \frac{\varphi(2^{-k}|\xi|)}{|\xi|^{\frac{d}{2}}}
\, \mathrm{d} \xi 
=
2^{\frac{d}{2}k}
\int_{\mathbb R^d} e^{i\theta_k(\xi)}
    \widetilde{\varphi}(|\xi|)
    \,\mathrm{d} \xi =: \mathcal{I}_k
\]
where $\widetilde{\varphi}(\xi) := |\xi|^{-\frac{d}{2}}\varphi(|\xi|)$, and $\theta_k(\xi) = 2^k x \cdot \xi + 2^{mk} t |\xi|^m$. Also, we split $V$ into
\[
V_1
=
\{
k\in V : 2^{mk}|t| < a2^k|x|\,\, \text{or} \,\, 2^{mk}|t| > b2^k|x|
\}
\]
and 
\[
V_2
=
\{
k\in V :  a 2^k|x| \leq 2^{mk}|t| \leq  b 2^k|x|
\}
\]
where the constants $a$ and $b$ are chosen as in the proof of Lemma \ref{l:disp space frec loc}. Since we have $|\nabla \theta_k(\xi)| \gtrsim 2^k|x|$ whenever $k \in V_1$ and $|\xi|$ belongs to the support of $\widetilde{\varphi}$, integration by parts yields
\[
|\mathcal{I}_k| \leq C_N 2^{\frac{d}{2}k}(2^k|x|)^{-N}
\]
for any natural number $N$. Thus, if we choose $N$ sufficiently large then
\[
\sum_{k \in V_1} |\mathcal{I}_k| \lesssim |x|^{-N} \sum_{2^k \geq |x|^{-1}} 2^{-k(N - \frac{d}{2})} \lesssim |x|^{-\frac{d}{2}}.
 \]
Next consider $k \in V_2$, and note that
\[
| \det \mathrm{Hess} \, \theta_k (\xi)| \sim (2^{km}|t|)^d \sim (2^k|x|)^d
\]
if $\xi$ belongs to the support of $\widetilde{\varphi}$. Again using standard results in the theory of oscillatory integrals, we obtain
\[
|\mathcal{I}_k| \lesssim 2^{\frac{d}{2}k} (2^k|x|)^{-\frac{d}{2}} \sim |x|^{-\frac{d}{2}}.
\]
However, the cardinality of $V_2$ is $O(1)$ and so this completes the proof.
\end{proof}

\begin{remarks}
(i) More generally, for $m\in(1,\infty)$ and any $\frac d2\leq\gamma<d$ we have
\begin{equation} \label{e:generalgamma}
\sup_{t\in\mathbb R}\left|\int_{\mathbb R^d} e^{i(x\cdot \xi+t|\xi|^m)}\frac{\varphi_{\leq N}(|\xi|)}{|\xi|^\gamma}\,\d\xi\right|
\lesssim
|x|^{-(d-\gamma)}
\end{equation}
uniformly in $N\in\mathbb Z$. The non-endpoint cases $\frac d2 < \gamma<d$ follow quickly from \eqref{e:Sjolin} and a dyadic decomposition. Moreover, we remark that \eqref{e:generalgamma} also holds for $m\in (0,1)$. 

(ii) Although it is not directly of use to us here, we note that, similar to Lemma \ref{l:disp space frec loc}, we have for $m\in(0,1)$, $\gamma<\frac d2$, and $|x|<1$
\[
\sup_{t\in(-1,1)}
\left|
\int_{\mathbb R^d} e^{i(x\cdot\xi+t|\xi|^m)}\frac{\varphi_{\leq N} (|\xi|)}{|\xi|^{\gamma}}\,\d\xi
\right|
\lesssim
|x|^{-(\frac d2+\frac{1}{1-m}(\frac d2-\gamma))}
\]
uniformly in $N\in \mathbb Z$.
\end{remarks}

Now we show \eqref{e:goal maximal} for $m>1$ and $s=\frac d4$. The main difference from the previous arguments to handle the delicacy here is to avoid the use of the decomposition in frequency. In fact, our goal \eqref{e:goal maximal} follows if we prove
\[
\|W\mathcal T_l\overline{W}\|_{\C^{\beta'}}
\lesssim
2^{(\frac{d}{2}-\frac\alpha\beta)l}\|W\|_{L^\infty(\d\mu)L^2}^2
\]
for $\beta\in[1,\frac{2\alpha}{d})$ and uniformly in $N$, where
\[
\mathcal T_lF(t,x)
=
\iint F(t',x')\mathcal K_l(t-t',x-x')\, \d t'\d \mu(x')
\]
and
\[
\mathcal K_l(t,x)
=
\chi_l(|x|)
\int_{\mathbb R^d}e^{i(x\cdot\xi+t|\xi|^m)}\frac{\varphi_{\leq N}(|\xi|)}{|\xi|^{\frac{d}{2}}}\,\d\xi.
\]

Since $\frac{2\alpha}{d}\leq2$, it is enough to consider the $\C^2$ norm and the $\C^\infty$ norm. Lemma \ref{l:disp space} yields the kernel estimate $|\mathcal K_l(t,x)|\lesssim \chi_l(|x|)2^{(d-2s)l}$. Hence, following a similar argument as before, we obtain
\[
\|W\mathcal T_l\overline{W}\|_{\C^2}
\lesssim
2^{(\frac{d}{2}-\frac \alpha2)l}\|W\|_{L^\infty(\d\mu)L^2}^2
\] 
and
\[
\|W\mathcal T_l\overline{W}\|_{\C^\infty}
\lesssim
2^{(\frac{d}{2}-\alpha)l}\|W\|_{L^\infty(\d\mu)L^2}^2,
\] 
and therefore
\[
\|W\mathcal T_l\overline{W}\|_{\C^{\beta'}}
\lesssim
2^{(\frac{d}{2}-\frac \alpha \beta)l}\|W\|_{L^\infty(\d\mu)L^2}^2.
\]
This implies \eqref{e:goal maximal} since $\frac d2<\frac\alpha\beta$.

\subsection{Necessary conditions} \label{subsection:necssary}

Here we establish two necessary conditions for the maximal-in-time estimate \eqref{maximal} to hold, and thus complete the proof of Theorem \ref{t:maximal main}. We begin with the more delicate case of $0 < m < 1$ (in which case we consider $d = 1$).

\subsection*{Condition 1}
Let $d=1$, $0<m<1$, and $0<\alpha\leq 1$. We show $1-2s\leq\frac m2 +(1-m)\frac\alpha\beta$ is necessary for \eqref{maximal} to hold for all orthonormal systems $(f_{j})_{j}$ in $\dot{H}^s(\mathbb R)$. 

Let $N \simeq 1$, which we take to be sufficiently large later in the proof. Define 
\begin{align*}
    I_j
    &=
    [2 N^{-\alpha(1-m)}j, 2N^{-\alpha(1-m)}j+N^{-(1-m)}],
\end{align*}
and $\J=\{j\in\mathbb Z : I_j\cap [-1,1]\not=\emptyset\}$. Then, we set 
\[
    \mu(x)
    =
    c N^{(1-\alpha)(1-m)}\sum_{j\in\J}\chi_{I_j}(x)\d x.
\]
Then we have $\mu(B(x,r)) \lesssim r^\alpha$ for any $x \in [-1,1], r>0$. To see this, first note that it clearly suffices to check the case $r \leq 2$. Now write $M = N^{1-m}$. In the case $2r<M^{-\alpha}$, a ball of radius $r$ contains at most one interval so that 
\[
M^{1-\alpha} \sum_{j \in\J}| B(x,r)\cap I_{j}|
\lesssim
M^{1-\alpha}\min\{r,M^{-1}\}
\lesssim
r^\alpha.
\]
On the other hand, if $\frac12 M^{-\alpha} \leq r \leq 2$, and denoting $r = \ell M^{-\alpha}$ with $\frac12 \leq \ell \leq 2 M^\alpha$, then number of $j \in \J$ such that $B(x,r) \cap I_j \not= \emptyset$ is bounded by $3 \ell$. Therefore, we have
\[
M^{1-\alpha}\sum_{j \in \J} |B(x,r)\cap I_{j}|\lesssim \ell M^{-\alpha} \lesssim r^{\alpha}.
\]
It follows that $\mu\in\mathcal M^\alpha([-1,1])$ if the constant $c>0$ is chosen appropriately. 

We choose the initial data as
\[
    f_j(x)
    =
    c' N^{\frac12-\frac m4}\varrho(N^{1-\frac m2}(x-2N^{-\alpha(1-m)}j)) e^{i(x-2jN^{-\alpha(1-m)})N}.
\]
Here, we choose $\varrho = \widehat{\psi} * \widehat{\psi}$, where $\widehat{\psi}$ is a standard bump function supported on $[-\frac{1}{2},\frac{1}{2}]$. Then $\varrho\in C_0^\infty(\R)$, $\supp\varrho\subset[-1,1]$, $\widehat{\varrho}$ is non-negative, and $\int_{[-1,1]^c} \widehat{\varrho}(\xi) \, \mathrm{d}\xi \leq  (1 - \varepsilon) \int_{[-1,1]} \widehat{\varrho}(\xi) \, \mathrm{d}\xi$ for some $\varepsilon \simeq 1$ sufficiently small. From the disjoint supports, it is easy to check that $(f_j)_j$ is orthonormal in $L^2$ upon an appropriate choice of the constant $c' > 0$. 

For each $j\in \J$ we claim that
\begin{equation} \label{e:necessarylowermain}
    |U_m |D|^{-s}f_j(t_j(x),x)|
    \gtrsim
    N^{\frac12-\frac m4-s}
\end{equation}
whenever $|x-2N^{-\alpha(1-m)}j|\leq \nu N^{-(1-m)}$ and $t_j(x)=-m^{-1}N^{1-m}(x-2N^{-\alpha(1-m)}j)$ (for this particular choice of $t_j(x)$, see also Figure \ref{f:counterxample}). Here, $\nu \simeq 1$ will be chosen sufficiently small later in the argument. We note, in particular, that $|t_j(x)| < 1$ is guaranteed as long as we take $\nu < m$.

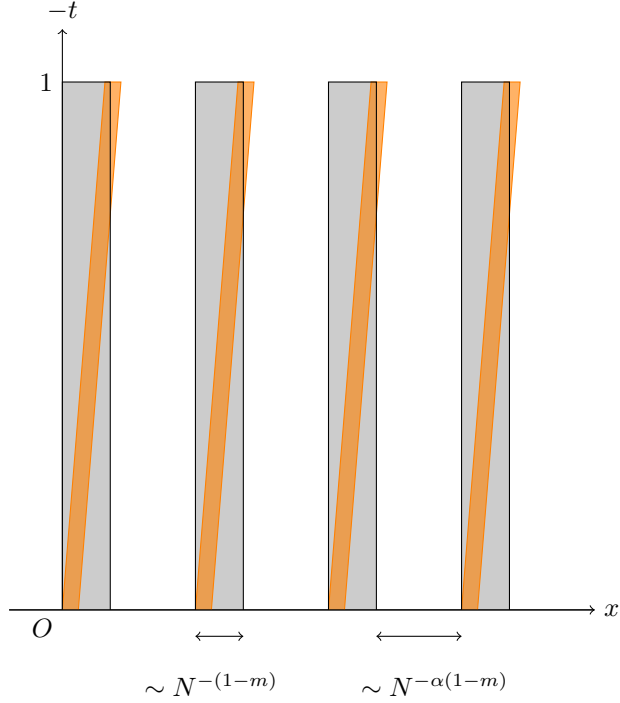
\begin{figure}[t]
\begin{center}
\begin{tikzpicture}[scale=7]
\coordinate (O) at (0,0);

\path[name path=line] (0,0)--(1,0.8);
\path[name path=column] (0.07,0)--(0.07,1);
\path[name intersections={of=line and column, by={A}}];

\draw [->] (-0.1,0)--(1,0);
\draw [->] (0,0)--(0,1.1);

\node [right] at (1,0) {$x$};
\node [above] at (0,1.1) {$-t$};
\node [below left] at (O) {$O$};
\node [left] at (0,1) {$1$};

\foreach \t in {0,1,...,3}{
\fill[gray,opacity=0.4] (0+0.25*\t,0)--(0+0.25*\t,1)--(0.09+0.25*\t,1)--(0.09+0.25*\t,0);
}

\foreach \t in {0,1,...,3}{
\fill[orange, opacity=0.6] (0+0.25*\t,0)--(0+0.25*\t+0.08,1)--(0.03+0.25*\t+0.08,1)--(0.03+0.25*\t,0);
}
\foreach \t in {0,1,...,3}{
\draw [orange](0+0.25*\t,0)--(0+0.25*\t+0.08,1)--(0.03+0.25*\t+0.08,1)--(0.03+0.25*\t,0);
}

\foreach \t in {0,1,...,3}{
\draw (0+0.25*\t,0)--(0+0.25*\t,1)--(0.09+0.25*\t,1)--(0.09+0.25*\t,0);
}

\node [below] at (0+0.25*1+0.03,-0.1) {$\sim N^{-(1-m)}$};

\node [below] at (0+0.25*2+0.2,-0.1) {$\sim N^{-\alpha(1-m)}$};

\draw [<->] (0.09+0.25*2,-0.05)--(0+0.25*3,-0.05);
\draw [<->] (0+0.25*1,-0.05)--(0.09+0.25*1,-0.05);


\draw (-0.1,0)--(1,0);

\end{tikzpicture}
\caption{The construction of initial data for Condition 1. The grey columns consist of their bases $I_j$ with the height $t$. Our choice of $t=t(x)$ belongs to the orange region.}
\label{f:counterxample}
\end{center}

\end{figure}

It suffices to check \eqref{e:necessarylowermain} when $j = 0$ since
\[
    \widehat{f_j}(\xi)
    =
    c' 
    N^{-(\frac12-\frac m4)}e^{-2ij\xi N^{-\alpha(1-m)}}\widehat{\varrho}(N^{-(1-\frac m2)}(\xi-N)). 
\]
Now 
\begin{align*}
    |U_m | D|^{-s}f_0(t,x)|
    &\simeq
    N^{-(\frac12-\frac m4)}\bigg|\int e^{i(x\xi+t|\xi|^m)}|\xi|^{-s}\widehat{\varrho}(N^{-(1-\frac m2)}(\xi-N))\,\d\xi\bigg|\\
    &=
 N^{\frac{1}{2} - \frac{m}{4} - s} \bigg|\int e^{i\theta(\xi)}|1 + N^{-m/2}\xi|^{-s} \widehat{\varrho}(\xi)\,\d\xi\bigg|,
\end{align*}
where $\theta(\xi)=xN^{1-\frac m2}\xi+N^mt|1+N^{-\frac m2}\xi|^m - N^mt$. If $|\xi| \leq 1$ then a Taylor expansion gives
\[
\theta(\xi) = (xN^{1-\frac m2} + mN^\frac{m}{2}t)\xi+ tO(|\xi|^2)
\]
and we observe that by choosing $t = t_0(x)$ the coefficient of $\xi$ vanishes. A careful check of the various constants reveals that $|\theta(\xi)| \leq 4m^{-1}\nu$ whenever $|x|\leq \nu N^{-(1-m)}$, $t = t_0(x)$ and $|\xi| \leq 1$. For such $x$ and $t$ we therefore have
\[
\bigg|\int_{-1}^1 e^{i\theta(\xi)}|1 + N^{-m/2}\xi|^{-s}\widehat{\varrho}(\xi)\,\d\xi \bigg| \geq (1 - \tfrac{\varepsilon}{4}) \int_{-1}^1 \widehat{\varrho}(\xi) \, \mathrm{d}\xi
\]
by taking $N$ sufficiently large, and taking $\nu$ sufficiently small. For the contribution for $|\xi| > 1$, by considering the cases $|\xi| \in [1,\delta N^\frac{m}{2}]$ and $|\xi| > \delta N^\frac{m}{2}$, and using the fast decay of $\widehat{\varrho}$ we have
\begin{align*}
\bigg|\int_{|\xi| > 1} |1 + N^{-m/2}\xi|^{-s}\widehat{\varrho}(\xi)\,\d\xi \bigg| & \leq \frac{1}{(1-\delta)^s} \int_{|\xi| > 1} \widehat{\varrho}(\xi) \, \mathrm{d}\xi + \frac{C}{\delta^2 N^{\frac{m}{2}}} \\
& \leq \frac{1-\varepsilon}{(1-\delta)^s} \int_{-1}^1 \widehat{\varrho}(\xi) \, \mathrm{d}\xi + \frac{C}{\delta^2 N^{\frac{m}{2}}}
\end{align*}
for some $C \simeq 1$. Taking $\delta$ sufficiently small, and then taking $N$ sufficiently large, we have
\[
\bigg|\int e^{i(x\xi+t|\xi|^m)}|\xi|^{-s}\widehat{\varrho}(N^{-(1-\frac m2)}(\xi-N))\,\d\xi\bigg| \geq  \frac{\varepsilon}{4} \int_{-1}^1 \widehat{\varrho}(\xi) \, \mathrm{d}\xi - \frac{C}{\delta^2 N^{\frac{m}{2}}} \geq   \frac{\varepsilon}{8} \int_{-1}^1 \widehat{\varrho}(\xi) \, \mathrm{d}\xi 
\]
which gives \eqref{e:necessarylowermain} when $j = 0$.

From \eqref{e:necessarylowermain} we obtain
\begin{align*}
    \bigg\|\sum_{j\in \J} |U_m|D|^{-s}f_j|^2\bigg\|_{L^1(\d\mu)L^\infty}
    &\geq
    \sum_{j\in\J} \int_{I_j} \sup_{t\in[-1,1]}|U_m|D|^{-s}f_j(t,x)|^2\,\d x\\
    &\gtrsim
    N^{1-\frac m2-2s}.
\end{align*}
If we assume that \eqref{maximal} is true, then the above yields $N^{1-\frac m2-2s} \lesssim (\#\J)^\frac1\beta \lesssim N^{(1-m)\frac\alpha\beta}$, and letting $N \to \infty$ we deduce that $1-2s\leq\frac m2 +(1-m)\frac\alpha\beta$ as desired. \qed

\subsection*{Condition 2} Let $d\in \mathbb N$,  $m\in(0,\infty)\backslash\{1\}$ and $0<\alpha\leq d$. We show $d-2s\leq \frac\alpha\beta$ is necessary for \eqref{maximal} to hold for all orthonormal systems $(f_j)_j$ in $\dot{H}^s(\mathbb R^d)$. 
Let $N \geq1$, and define
\[
    I_j
    =
    [2N^{-\frac{\alpha}{d}}j, 2N^{-\frac{\alpha}{d}}j+100^{-1}N^{-1}]\subset \mathbb R
\]
for $j \in \Z$ and $\J=\{(j_1,\dots,j_d)\in\mathbb Z^d: I_{j_1}\times\cdots\times I_{j_d}\cap\mathbb B^d\not=\emptyset\}$. Then, we set 
\[
    \mu(x)
    =
    cN^{d-\alpha}\sum_{\J}\chi_{I_{j_1}\times\cdots\times I_{j_d}}(x)\d x.
\]
By a similar argument as we used for Condition 1, one can easily verify that $\mu(B(x,r)) \lesssim r^\alpha$ for any $x \in \mathbb{B}^d, r>0$. Hence, with an appropriate constant $c>0$, we have $\mu\in\mathcal M^\alpha(\mathbb B^d)$.

For ${\bf{j}} \in \Z^d$, we define the initial data
\[
    f_{\bf{j}}(x)
    =
    c' N^\frac d2\varrho(N(x-2N^{-\frac{\alpha}{d}}{\bf{j}})).
\]
Here, similar to Condition 1, we choose $\varrho\in C_0^\infty (\R^d)$ such that $\supp\varrho\subset \mathbb B^d$, $\widehat{\varrho}$ is non-negative, and $\int_{|\xi|\geq 1}\widehat{\varrho}(\xi)\,\d\xi \leq (1 - \varepsilon) \int_{\mathbb B^d}\widehat{\varrho}(\xi)\,\d\xi$ for some $\varepsilon \simeq 1$. It is simple to check that $(f_{\bf{j}})_{\bf{j}}$ is orthonormal in $L^2$ if we make an appropriate choice of the constant $c'>0$. Also, we claim that, for each ${\bf{j}}\in\J$, we have
\[
    |U_m | D|^{-s}f_{\bf{j}}(t,x)|
    \gtrsim
    N^{\frac d2-s}
\]
whenever $|x-2N^{-\frac{\alpha}{d}}{\bf{j}}|\leq \nu N^{-1}$ and $|t| \leq \nu N^{-m}$ ($\nu \simeq 1$ to be chosen sufficiently small). To see this, since
\[
    \widehat{f_{\bf{j}}}(\xi)
    =
    c' N^{-\frac d2}e^{-2iN^{-\frac{\alpha}{d}}{\bf{j}}\cdot\xi}\widehat{\varrho}(N^{-1}\xi),
\]
it suffices to check the case ${\bf{j}}=0$. Now
\begin{align*}
    |U_m | D|^{-s} f_0(t,x)|
    \simeq 
    N^{\frac d2 -s}\bigg|\int e^{i\theta(\xi)}|\xi|^{-s} \widehat{\varrho}(\xi)\,\d\xi\bigg|,
\end{align*}
where the phase is denoted by $\theta(\xi)=Nx\cdot \xi+N^mt|\xi|^m$. If $|x|\leq \nu N^{-1}$ and $|t|\leq \nu N^{-m}$, then the phase is sufficiently small so that 
\begin{align*}
\bigg|\int_{\mathbb{B}^d} e^{i\theta(\xi)}|\xi|^{-s} \widehat{\varrho}(\xi)\,\d\xi\bigg|
&\geq
(1 - \tfrac{\varepsilon}{2}) \int_{\mathbb{B}^d} \widehat{\varrho}(\xi)\,\d\xi.
\end{align*}
The contribution for $|\xi| \geq 1$ can be easily estimated from above using the properties of $\widehat{\varrho}$, and the claim follows.

From the above we conclude that
\begin{align*}
    \bigg\|\sum_{{\bf{j}} \in \J}|U_m |D|^{-s}f_{\bf{j}}|^2\bigg\|_{L^1(\d\mu)L^\infty}
    &\geq
    \int_{\mathbb{B}^d} \sup_{t\in [-1,1]}\sum_{{\bf{j}}\in \J}|U_m |D|^{-s}f_{\bf{j}}(t,x)|^2\,\d\mu(x)\\
    &\gtrsim
    N^{d-2s}.
\end{align*}
This means that if \eqref{maximal} is true, then  we obtain $N^{d-2s}\lesssim (\#\J)^\frac1\beta \lesssim N^\frac\alpha\beta$, and letting $N \to \infty$ we deduce that $d-2s \leq \frac{\alpha}{\beta}$ as desired. \qed

\begin{Acknowledgements}
The first author would like to express his thanks to Shohei Nakamura and Sanghyuk Lee for many inspiring conversations related to this content of this paper. Part of this work was carried out whilst the authors were participating in the MATRIX--RIMS Tandem Workshop on Geometric Analysis in Harmonic Analysis and PDE at RIMS during 27--31 March 2023, and the authors are grateful for the stimulating working environment.
\end{Acknowledgements}



\begin{thebibliography}{MM}

\bibitem{BBCR} J. A. Barcel\'{o}, J. Bennett, A. Carbery, K. M. Rogers, 
\textit{On the dimension of divergence sets of dispersive equations}, 
Math. Ann. \textbf{349} (2011), 599--622. 

\bibitem{BL76} J. Bergh and J. L\"{o}fstr\"{o}m, \textit{Interpolation Spaces}, Springer, Berlin, 1976.

\bibitem{BHLNS} N. Bez, Y. Hong, S. Lee, S. Nakamura, Y. Sawano, \textit{On the Strichartz estimates for orthonormal systems of initial data with regularity},  Adv. Math. \textbf{354} (2019), 106736, 37 pp. 

\bibitem{BKS_RIMS} N. Bez, S. Kinoshita, S. Shiraki 
\textit{A note on Strichartz estimates for the wave equation
with orthonormal initial data},
preprint.

\bibitem{BLN_Selecta} N. Bez, S. Lee, S. Nakamura, \textit{Maximal estimates for the Schr\"odinger equation with orthonormal initial data}, Selecta Math. \textbf{26} (2020), Article 52.

\bibitem{BLN_Forum} N. Bez, S. Lee, S. Nakamura, \textit{Strichartz estimates for orthonormal families of initial data and weighted oscillatory integral estimates}, Forum of Math. Sigma, \textbf{9} (2021), 52pp.

\bibitem{Bourgaintrick}
J. Bourgain,  
\textit{Estimations de certaines fonctions maximales}, 
C. R. Acad. Sci. Paris S\'{e}r. I Math.
\textbf{310} (1985), 499--502.


\bibitem{Bourgain} J. Bourgain, \textit{A note on the Schr\"{o}dinger maximal function}, J. Anal. Math. \textbf{130} (2016), 393--396.

\bibitem{Cal64} A. P. Calder\'{o}n, \textit{Intermediate spaces and interpolation, the complex method}, Studia Math. \textbf{24} (1964), 113--190.

\bibitem{Carleson} L. Carleson, \textit{Some analytic problems related to statistical mechanics, in Euclidean Harmonic Analysis (Proc. Sem., Univ. Maryland, College Park, Md., 1979)}, 5--45, Lecture Notes in Math. \textbf{779}, Springer, Berlin.







\bibitem{CHP1} T. Chen, Y. Hong, N. Pavlovi\'{c}, \textit{Global well-posedness of the NLS system for infinitely many fermions}, Arch. Ration. Mech. Anal. \textbf{224} (2017), 91--123.

\bibitem{CHP2} T. Chen, Y. Hong, N. Pavlovi\'{c}, \textit{On the scattering problem for infinitely many fermions in dimension $d\geq3$ at positive temperature}, Ann. Inst. H. Poincar\'e Anal. Non Lin\'eaire \textbf{35} (2018), 393--416.

\bibitem{CK18}
C. H. Cho,
H. Ko,
\textit{Pointwise convergence of the fractional Schr\"odinger equation in $\mathbb R^2$}, Taiwanese J. Math. 26 (2022), no. 1, 177--200.

\bibitem{CL}
C. H. Cho, S. Lee, \textit{Dimension of divergence sets for the pointwise convergence of the Schr\"odinger equation}. J. Math. Anal. Appl. \textbf{411} (2014), 254--260.

\bibitem{ChoShiraki} C. H. Cho, S. Shiraki, \textit{Pointwise convergence along a tangential curve for the fractional {S}chr\"{o}dinger equation}, Ann. Fenn. Math. \textbf{46} (2021), 993--1005.

\bibitem{CS22}
C. H. Cho, S. Shiraki
\textit{Dimension of divergence sets of oscillatory integrals with concave phase},
arXiv:2212.14330.

\bibitem{Cowling}
M. Cowling, \textit{Pointwise behavior of solutions to Schr\"odinger equations}, In: Harmonic Analysis (Cortona, 1982), Lecture Notes in Math. 992, 83--90 (1983).



\bibitem{CJ} M. Cwikel, S. Janson, \textit{Interpolation of analytic families of operators}, Studia Math. \textbf{79} (1984), 61--71.

\bibitem{DK} B. E. J. Dahlberg, C. E. Kenig, \textit{A note on the almost everywhere behavior of solutions to the Schr\"{o}dinger equation, in Harmonic Analysis (Minneapolis, Minn., 1981)}, 205--209, Lecture Notes in Math. \textbf{908}, Springer, Berlin.

\bibitem{DuGuthLi} X. Du, L. Guth, X. Li, \textit{A sharp Schr\"{o}dinger maximal estimate in $\mathbb{R}^2$}, Ann. of Math. \textbf{186} (2017), 607--640.
.

\bibitem{DuZhang} X. Du, R. Zhang, \textit{Sharp $L^2$ estimate of Schr\"{o}dinger maximal function in higher dimensions}, Ann. of Math. \textbf{189} (2019), 837--861.


\bibitem{EP22}
D. Eceizabarrena, F. Ponce-Vanegas,
\textit{Pointwise convergence over fractals for dispersive equations with homogeneous symbol},
J. Math. Anal. Appl.
\textbf{515}
(2022),
126385.


\bibitem{Frank_ICM} R. Frank, \textit{Lieb--Thirring inequalities and other functional inequalities for orthonormal systems}, Proceedings of the ICM 2022 (submitted), arXiv:2109.13660.

\bibitem{Frank_9} R. Frank, \textit{The Lieb--Thirring inequalities: recent results and open problems}, Nine mathematical challenges — an elucidation, 45--86, Proc. Sympos. Pure Math., 104, Amer. Math. Soc., Providence, RI, 2021.

\bibitem{FGL} R. Frank, D. Gontier, M. Lewin, \textit{The nonlinear Schr\"odinger equation for orthonormal functions II: Application to Lieb--Thirring inequalities}, Comm. Math. Phys. \textbf{384} (2021), 1783--1828.

\bibitem{FLLS} R. Frank, M. Lewin, E. H. Lieb, R. Seiringer, \textit{Strichartz inequality for orthonormal functions}, J. Eur. Math. Soc. \textbf{16} (2014), 1507--1526.

\bibitem{FS_AJM} R. Frank, J. Sabin, \textit{Restriction theorems for orthonormal functions, Strichartz inequalities, and uniform Sobolev estimates}, Amer. J. Math. \textbf{139} (2017), 1649--1691.

\bibitem{FS_Survey} R. Frank, J. Sabin, \textit{The Stein--Tomas inequality in trace ideals}, S\'eminaire Laurent Schwartz -- EPD et applications (2015-2016), Exp. No. XV, 12 pp., 2016.

\bibitem{FrankSabin_Adv}  R. Frank, J. Sabin, \textit{Spectral cluster bounds for orthonormal systems and oscillatory integral operators in Schatten spaces}, Adv. Math. \textbf{317} (2017), 157--192. 

\bibitem{GLN}  D. Gontier, M. Lewin, F. Q. Nazar, \textit{The nonlinear Schr\"odinger equation for orthonormal functions: existence of ground states}, Arch. Ration. Mech. Anal. \textbf{240} (2021), 1203--1254.

\bibitem{Grafakos} L. Grafakos, \textit{Classical Fourier analysis}, third edition, Graduate Texts in Mathematics, vol. 249, Springer, New York, 2014.

\bibitem{GM14} L. Grafakos, M. Masty\l o, \textit{Analytic families of multilinear operators}, Nonlinear Anal. \textbf{107} (2014), 47--62.

\bibitem{GO22} L. Grafakos, E. M. Ouhabaz, \textit{Interpolation for analytic families of multilinear operators on metric measure spaces}, Studia Math. \textbf{267} (2022), 37--57.

\bibitem{GLNY} Z. Guo, J. Li, K. Nakanishi, L. Yan, \textit{On the boundary Strichartz estimates for wave and Schr\"{o}dinger equations}, J. Differential Equations \textbf{265} (2018), 5656--5675.


\bibitem{IP}
A. Ionescu, F. Pusateri, \textit{Nonlinear fractional Schrödinger equations in one dimension}, J. Funct. Anal. \textbf{266} (2014),
139--176.

\bibitem{Karpman1}
V. I. Karpman, \textit{Stabilization of soliton instabilities by higher-order dispersion: Fourth order nonlinear Schr\"odinger-type
equations}, Phys. Rev. E \textbf{53} (1996), 1336--1339.

\bibitem{Karpman2} V. I. Karpman, A. G. Shagalov, \textit{Stability of soliton described by nonlinear Schr\"odinger-type equations with higher-order
dispersion}, Phys. D \textbf{144} (2000), 194--210.

\bibitem{KeelTao} M. Keel, T. Tao \textit{Endpoint Strichartz estimates}, Amer. J. Math. \textbf{120} (1998), 955--980.

\bibitem{KPV} C. E. Kenig, G. Ponce, L. Vega, \textit{Oscillatory integrals and regularity of dispersive equations}, Indiana Univ. Math. J. \textbf{40} (1991), 33--69.

\bibitem{Laskin1} 
N. Laskin, \textit{Fractional quantum mechanis and Lévy path integrals}, Phys. Lett A \textbf{268} (2000), 298--305.

\bibitem{Laskin2}
N. Laskin, \textit{Fractional Schrödinger equation}, Phys. Rev. E \textbf{66}, 056108 (2002).

\bibitem{LewinSabin1} M. Lewin, J. Sabin, \textit{The Hartree equation for infinitely many particles. I. Well-posedness theory}, Comm. Math. Phys. \textbf{334} (2015), 117--170.

\bibitem{LewinSabin2} M. Lewin, J. Sabin, \textit{The Hartree equation for infinitely many particles. II. Dispersion and scattering in 2D}, Anal. PDE \textbf{7} (2014), 1339--1363.

\bibitem{Lieb_Sobolev} E. H. Lieb, \textit{An $L^p$ bound for the Riesz and Bessel potentials of orthonormal functions}, J. Funct. Anal. \textbf{51} (1983), 159--165.

\bibitem{Lieb_BAMS} E. H. Lieb, \textit{The stability of matter: from atoms to stars}, Bull. Amer. Math. Soc. \textbf{22} (1990), 1--49.

\bibitem{LiebThirring} E. H. Lieb, W. Thirring, \textit{Bound on kinetic energy of fermions which proves stability of matter}, Phys. Rev. Lett. \textbf{35} (1975), 687--689.


\bibitem{LR19a}
R. Luc\`{a},
K. M. Rogers,
\textit{A note on pointwise convergence for the Schr\"odinger equation},
Math. Proc. Cambridge Philos. Soc. \textbf{166} (2019), 209--218.

\bibitem{LR19b}
R. Luc\`{a},
K. M.  Rogers, 
\textit{Average decay of the Fourier transform of measures with applications},
J. Eur. Math. Soc. (JEMS) \textbf{21} (2019), 465--506.



\bibitem{MontSmith} S. J. Montgomery-Smith, \textit{Time decay for the bounded mean oscillation of solutions of the Schr\"odinger and wave equations}, Duke Math. J. \textbf{91} (1998), 393--408.

\bibitem{Nakamura} S. Nakamura, \textit{The orthonormal Strichartz inequality on torus}, Trans. Amer. Math. Soc. \textbf{373} (2020), 1455--1476.

\bibitem{Nguyen} N. Nguyen, \textit{Fermionic semiclassical $L^p$ estimates}, arXiv:2205.00722.

\bibitem{ONeil} R. O'Neil, \textit{Convolution operators and $L(p, q)$ spaces}, Duke Math. J. \textbf{30} (1963), 129--142.

\bibitem{Pr20}
L. B. Pierce, 
\textit{On Bourgain’s Counterexample for the Schr\"{o}dinger Maximal Function}, 
Q. J. Math. \textbf{71} (2020), 1309--1344.

\bibitem{Sabin_survey} J. Sabin, \textit{The Hartree equation for infinite quantum systems}, Journ\'{e}es \'{e}quations aux d\'{e}riv\'{e}es partielles, (2014), Exp. No. 8. 18p.

\bibitem{Simon} B. Simon, \textit{Trace ideals and their applications}, Vol. 35 of London Mathematical Society Lecture Note Series, Cambridge University Press, Cambridge, 1979.


\bibitem{SS89}
P. Sj\"ogren,
P. Sj\"olin,
\textit{Convergence properties for the time dependent Schr\"odinger equation},
Ann. Acad. Sci. Fenn. A I Math. \textbf{14} (1989), 13--25.

\bibitem{Sj87}
P. Sj\"olin,
\textit{Regularity of solutions to the Schr\"odinger equation},
Duke Math. J. \textbf{55} (1987), 699--715.

\bibitem{Sjolin}
P. Sj\"olin, \textit{Macimal estimates for solutions to nonelliptic Schr\"odinger equation,} Bull. Load. Math. Soc. \textbf{39} (2007), 404--412.

\bibitem{Stein56} E. M. Stein, \textit{Interpolation of linear operators}, Trans. Amer. Math. Soc. \textbf{83} (1956), 482--492.

\bibitem{Stein} E. M. Stein, \textit{Harmonic analysis: real-variable methods, orthogonality, and oscillatory integrals}, Princeton Mathematical Series, 43, Princeton University Press, 1993.

\bibitem{Wl97}
B. G. Walther, \textit{Maximal estimates for oscillatory integrals with concave phase}, Harmonic analysis and operator theory, 485--495, Contemp. Math. \textbf{189}, Amer. Math. Soc., Providence, RI, 1995. 

\end{thebibliography}
\end{document}